\documentclass[a4paper,american,reqno]{amsart}
\usepackage{float}
\usepackage{subcaption}
\usepackage{caption}
\newif\ifPreprint \Preprintfalse
\newif\ifSubmission \Submissiontrue

\usepackage{babel}
\usepackage[utf8]{inputenc}
\usepackage[T1]{fontenc}
\usepackage{lmodern}
\usepackage{xspace}
\usepackage{mathtools}
\usepackage{amsmath}
\usepackage{enumitem}
\usepackage{accents}
\usepackage[disable]{todonotes}
\usepackage{csquotes}
\usepackage{microtype}
\usepackage{pgfplots}
\usepackage{longtable}
\usepackage{booktabs}
\pgfplotsset{compat=1.18}
\usetikzlibrary{positioning,arrows.meta}

\usepackage[style  = numeric-comp,
            firstinits = true,
            maxbibnames = 100,
            maxcitenames = 2,
            isbn = false,
            backend = bibtex]{biblatex}
\usepackage[colorlinks,
            citecolor=blue,
            urlcolor=blue,
            linkcolor=blue]{hyperref}
\usepackage{cleveref}
\usepackage{bbm}

\makeatletter
\patchcmd{\@settitle}{\uppercasenonmath\@title}{\scshape\large}{}{}
\patchcmd{\@setauthors}{\MakeUppercase}{\scshape\normalsize}{}{}
\makeatother

\vfuzz \hfuzz
\theoremstyle{plain}
\newtheorem{theorem}{Theorem}
\newtheorem{lemma}[theorem]{Lemma}
\newtheorem{coroll}[theorem]{Corollary}

\theoremstyle{definition}

\newtheorem{example}[theorem]{Example}

\newtheorem{remark}[theorem]{Remark}

\theoremstyle{remark}

\numberwithin{theorem}{section}

\newcommand{\st}{\text{s.t.}}

\newcommand{\R}{\mathbb{R}}

\newcommand{\Ccal}{\mathcal{C}}
\newcommand{\Mcal}{\mathcal{M}}

\DeclareMathOperator*{\argmin}{arg\,min}

\bibliography{bibliography}

\begin{document}

\title[A Safe Approximation of Univariate DRO]{
  Safe Mixed-Integer Approximation for Non-Convex Distributional Robustness with Univariate Indicator Functions}
\author[J. Dienstbier, F. Liers, F. R{\"o}sel, J. Rolfes]%
{J. Dienstbier, F. Liers, F. R{\"o}sel, J. Rolfes}

\address[J. Dienstbier, F. Liers, F. R{\"o}sel]{%
  Department of Data Science,
  Friedrich-Alexander-Universität, Erlangen-Nürnberg,
  Cauerstr. 11,
  91058 Erlangen,
  Germany}
\email{\{jana.jd.dienstbier, frauke.liers, florian.roesel\}@fau.de}

\address[J. Rolfes]{%
	Department of Mathematics,
	Linköping University,
	SE-581 83 Linköping,
	Sweden}
\email{\{jan.rolfes\}@liu.se}

\newcommand{\RTD}[1]{\rho({#1})}
\newcommand{\uRTD}[2]{\rho_{#1}({#2})}
\newcommand{\mass}[1]{\mathbbm{P}}
\newcommand{\EW}[1]{\mathbb{E}_{#1}}
\newcommand{\varianz}[1]{\sigma^{#1}}
\newcommand{\bound}[2]{\bar{\rho}_{#1}({#2})}
\newcommand{\initial}[1]{q_0({#1})}
\newcommand{\boundvariable}[1]{z_{#1}}
\newcommand{\lift}[2]{\Delta_{#1}^{#2}}
\newcommand{\minparabel}{p_{\text{min}}}
\newcommand{\sign}{\text{sign}}
\newcommand{\E}{\mathbbm{E}}
\renewcommand{\P}{\mathbbm{P}}
\renewcommand{\u}{\mathcal{P}}
\newcommand{\Pcal}{\mathcal{P}}

\newcommand{\change}{\textcolor{red}}
\newcommand{\Lcal}{\mathcal{L}}

\date{\today}

\begin{abstract}
In this work, we present an algorithmically tractable safe approximation of distributionally robust optimization (DRO)
problems that contain univariate indicator functions. The latter appear in different applications, but render the model
nonlinear and nonconvex. The considered ambiguity sets can exploit moment information. Typically, reformulation approaches using duality theory need to
make strong assumptions on the structure of the underlying
constraints, such as convexity in the decisions or concavity in the
uncertainty which cannot be assumed in our setting. We nevertheless present a semi-infinite minorant of the adversarial problem that admits a zero-gap dualization and is subsequently approximated by a discretized counterpart. Under mild assumptions, the latter provides a safe approximation that is a mixed-integer linear problem and can be solved by available standard software.
Obtained solutions are guaranteed to be feasible for the original distributionally robust problem.
We prove that the approximation of the adversarial problem converges to its true value for increasingly fine discretization, and that the values of the solutions of the safe approximations converge to the value of the true DRO problem under mild assumptions.
In addition, the error bound yields an a posteriori enclosure of the true optimal value that can be evaluated in practice.
On the practical side, the approach is made concrete for a challenging, fundamental task in material design, namely in particle separation.
Computational results for a realistic setting show that the safe approximation meets the imposed quality requirement for every distribution in the ambiguity set, whereas the unprotected solution does not.

\end{abstract}

\keywords{%
Distributionally Robust Optimization,
Reformulation,
Safe Approximation,
Mixed-Integer Optimization,
Robust Optimization,
Stochastic Optimization,
Discrete Optimization%
}

\makeatletter
\@namedef{subjclassname@2020}{%
	\textup{2020} Mathematics Subject Classification}
\makeatother

\subjclass[2020]{%
90Cxx,
90C11,
90C22
90C34%
}

\maketitle

\section{Introduction}
\label{Sec:introduction}

Structural and algorithmical insights in mathematical optimization have led to the development of many practically efficient methodologies.
Additionally, optimization under uncertainty is a research area that currently
receives increased attention. Real-world applications are often
strongly affected by uncertainties that can
neither be fully controlled nor be measured exactly. In this work,
we study a class of distributionally robust optimization models that contain univariate indicator functions. Next to classical quantities such as the value at risk, univariate indicator functions appear in relevant and challenging applications. In particular, they can be used to model the decision for an optimized time interval during which energy generation is performed (or switched off again) in the unit commitment context in the power industry, where energy generation is uncertain for renewables,~\cite{unitcommitment}.
Another application consists in deciding for a best possible time interval during which material shall be collected in particle separation under unavoidable uncertainties. However, as indicator functions lead to non-convex models, general and practically efficient robust algorithms are not available. In this work, we present a safe approximation for this class of problems. Its usage is advantageous because it yields solutions that are guaranteed to be robust feasible and that can be obtained by solving a mixed-integer linear optimization problem (MIP) with available modern software tools.

Let $b\in \R$ be a scalar, $\u$ a set of probability measures , $x\in
\R^k$ decision variables
We model the uncertainty in our distributionally robust optimization (DRO) model by a random variable $t$ with values in the compact set $T \subseteq \R$ distributed according to an uncertain probability measure $\P\in \u$.
Here, $v:\R^k\times T\rightarrow \R$ denotes a function that connects the decision variables with the random variable $t$. Then, a \emph{distributionally robust constraint} or DRO constraint is defined by

\begin{equation}\label{Eq: DRO_constraint}
		b\leq \min_{\P\in \u} \E_{\P}\left(v(x,t)\right),
\end{equation}

where $\E_{\P}$ models the expected value with respect to $\P$.
Constraints of this form have been intensively studied for decades, see e.g., \cite{Fengmin2021a} or \cite{Rahimian2019a}. In
particular, setting $\u=\{\P\}$ \emph{stochastic
  constraints}
\begin{equation*}
	b\leq \E_{\P}\left(v(x,t)\right)
\end{equation*}
are a central object of investigation in stochastic
optimization.

On the other hand, setting $\u=\{\delta_t: t\in T\}$, where $\delta_t$ denotes the Dirac measures at $t\in T$, \eqref{Eq: DRO_constraint} provides a \emph{robust constraint}
\begin{equation*}
	b\leq \min_{t\in T} v(x,t).
\end{equation*}
A constraint of this form is a
central object of interest in robust optimization, where challenges still consist in deriving algorithmically tractable
solution approaches for solving the robust counterparts.
Typically, reformulations, decomposition and approximation algorithms are in the focus.
In this work, we develop a safe approximation for a DRO constraint, i.e., obtained solutions are guaranteed to be robust, that is allowed to contain indicator functions. Indicator functions are very relevant as they can be used to approximate arbitrary functions $v$ and consequently arbitrary Lebesgue integrals such as $\EW{\P}\left(v(x,t)\right).$ Although indicator functions are non-convex, the safe approximation is obtained via duality-based reformulations of some closely related (convex) problems.
Next,
we briefly review some relevant literature in optimization under uncertainty, and distributional robustness in particular.
\bigskip

If the uncertainty is distributed by a given distribution or if the
constraints need to be satisfied in a probabilistic sense, stochastic
optimization is useful (see e.g., \cite{prekopa1998SO,birge2006introduction}).
In case underlying distributions are unknown, robust optimization assumes
that uncertainty is modeled via predefined uncertainty sets. Solutions are sought that are feasible regardless of how the uncertainties manifest themselves within these sets. Robust optima are those with best guaranteed solution value. Duality-based
reformulation approaches
have been developed for a wide class of optimization problems, achieving computational
tractability of semi-infinite or at least exponentially large
robust counterparts, see, e.g., \cite{soyster1973convex,ben1998robust,ben1999robust,ben2000robust}.
If reformulations are not available, often decomposition as well as approximation approaches are developed.
For textbooks on robust optimization, we refer to~\cite{Ben-Tal:RobustOptimization} and~\cite{bertsimas2022robust} for continous robust optimization and to~\cite{kouvelis97} and~\cite{GoerigkHartisch24} for discrete and combinatorial robust optimization. Still, robust optimization including non-convex functions pose methodological and algorithmical challenges. A decomposition algorithm based on an adaptive bundle method was
presented in \cite{Kuchlbauer2022adaptive} for nonlinear robust optimization and extended to
integral decisions in \cite{kuchlbauerOuter}. An overview of robust nonlinear optimization is given in~\cite{Leyffer02042020}.

If the distribution itself is uncertain, determining distributionally robust
optima as shown in~\eqref{Eq: DRO_constraint} is natural.
For recent surveys, we refer to the detailed reviews \cite{Fengmin2021a,Rahimian2019a}, as well as to~\cite{Kuhn_Shafiee_Wiesemann_2025} and the references therein.
To summarize, there are two main strands of research to describe uncertain
distributions against which robust protection is sought. The latter are assumed to reside in a so-called \emph{ambiguity set}. In \emph{discrepancy-based ambiguity sets}, it is assumed that a 'typical', nominal distribution is given. The corresponding ambiguity set is then assumed to contain all distributions within an appropriately defined distance from the nominal distribution, e.g., Wasserstein-balls, see \cite{mohajerin2018data}. In the second strand, ambiguity sets are defined as
\emph{moment-based ambiguity sets}, where the moments of the considered distributions are assumed to satisfy given
bounds. We will use closely related ambiguity sets here and hence review further literature for moment-based ambiguities next.

In \cite{ghaoui2003worst} the authors study mean-variance or
Value-at-Risk risk measures, where the ambiguity set has a bounded mean and covariance matrix. The authors
develop algorithmically tractable algorithms for the robust counterpart. \cite{Popescu2007a} also considers moment information
where the problem is reduced to optimizing an univariate mean-variance robust objective.
In \cite{Xu2017a} Slater conditions are used to show correctness of a
duality-based reformulation of the robust counterpart. Additionally, two
discretization schemes are presented that provide convergent approximate solutions to the problem. Exact reformulations of DRO problems, that contain \eqref{Eq: DRO_constraint} as a constraint, are only known for specific ambiguity sets. \cite{Delage2010a} present such reformulations for
convex problems and ambiguity sets with moment bounds.
Moreover, the article \cite{Wiesemann2014a}, under some convexity assumptions, presents a tractable duality-based reformulation of \eqref{Eq: DRO_constraint}, that incorporates information on the confidence sets. Under the given assumptions, their approach can be applied to a DRO with \eqref{Eq: DRO_constraint} as a constraint.

Incorporating additional information to moment-based ambiguity sets forms a challenging task. One such additional property that is potentially beneficial to capture in an ambiguity set is if the probability distribution is known to be unimodal. Here, the authors of \cite{Parys2015a} determine a semidefinite programming (SDP) reformulation for the DRO constraint \eqref{Eq: DRO_constraint}, but do not consider its interplay with the upper-level decision variables $x$ encoded by $v(x,t)$.

Most of the current literature focuses on convex problems, and nonconvex classes of DRO problems have only rarely been studied. Among them,~\cite{JinEtAl2021} as well as~\cite{NeufeldEtAl2025} have developed first-order solution algorithms.

We use moment-based ambiguity sets and assume that the mean and the variance are within appropriately chosen bounds. In
addition, we allow to include confidence set information, as has been done in
\cite{Wiesemann2014a}. However, \cite{Wiesemann2014a} requires the models to be convex, whereas here we address nonconvexities as they appear in indicator
functions.
For the resulting semi-infinite DRO problem, we introduce a safe approximation
which approximates the robust counterpart by a finite-dimensional
mixed-integer linear optimization model. A main result consists in the
proof that the inner, 'adversarial', optimization problem can be approximated arbitrarily well, where the solution quality depends on the size of some discretization of the feasible domain. Computational experiments for a challenging application in particle design show that the resulting approach determines robust solutions quickly and that their distance to the optimal value of the original problem can be assessed a posteriori.
For the more general case of multivariate distributionally robust constraints that are allowed to contain indicator functions, a positive-semidefinite safe approximation was derived in~\cite{rolfes25}.

Current research aims at defining ambiguity sets that are as large as necessary, while restricting the level of conservatism that is induced by robust protection. The textbook~\cite{bertsimas2022robust} on robust optimization also covers recent research on appropriate choices of uncertainty and ambiguity sets. An approach in this direction consists in studying 'hybrid' ambiguities derived by intersecting moment-based and discrepancy-based ambiguity sets.

Several works derive approximations for DRO problems with hybrid ambiguities. For example, \cite{Luo2023} study distributionally robust chance constraints with  Wasserstein balls intersected with moment information. The safe approximation consists in solving a MIP. In~\cite{Cheramin2022}, both inner and outer approximations were derived for DRO problems with similar ambiguity sets.

Before we describe the structure of the article, we summarize the chain of problems that leads from the distributionally robust model to the mixed-integer program, since all following sections are organized along it. The adversarial problem~\eqref{problem:inner} is, in a first step, replaced by a minorant, i.e., by a problem whose optimal value is a lower bound: the discontinuous integrand is underestimated by a continuous function and the pointwise envelope constraint is replaced by finitely many constraints. This minorant is dualized without a duality gap, and its semi-infinite dual constraint is replaced by an equivalent finite system. Restricting this system makes it linear and lowers the optimal value once more; its re-dualization is a linear program over measures supported on a grid.
Theorem~\ref{thm:inner_convergence} bounds the total error accumulated in the two inexact steps. Finally, the remaining indicator functions are encoded by binary variables in the direction that preserves the bound. Every step either preserves the optimal value or decreases it, so that feasibility for the final model implies feasibility for the original one. Figure~\ref{fig:roadmap} displays this chain together with the direction of the respective estimate and the place where it is established.

\begin{figure}[ht]
	\centering
	\begin{tikzpicture}[
		node distance=11mm,
		box/.style={draw, rounded corners, align=center, inner sep=4pt, font={\small\hyphenpenalty=10000\relax}},
		head/.style={align=center, font=\small\bfseries},
		lab/.style={align=left, font=\footnotesize},
		arr/.style={-{Latex[length=2mm]}, thick},
		link/.style={-{Latex[length=2mm]}, thick, dashed}
		]
		\node[box, text width=0.28\textwidth] (a) {adversarial problem \eqref{problem:inner}\\ \footnotesize nonconvex integrand, infinite-dimensional, semi-infinite};
		\node[box, text width=0.28\textwidth, below=of a] (b) {minorant \eqref{Prob: Primal_Purity_Constraint_envelope}\\ \footnotesize continuous underestimator of the indicator, finitely many envelope constraints};
		\node[box, text width=0.28\textwidth, below=of b] (c) {semi-infinite dual \eqref{Prob: Dual_Purity_Constraint}, finite system \eqref{constr:finite}};
		\node[box, text width=0.28\textwidth, below=of c] (d) {restricted dual \eqref{Prob: Dual_Purity_Constraint_ref} and its re-dualization \eqref{problem:safe_approx_primal}\\ \footnotesize linear program over grid measures};
		\node[box, text width=0.24\textwidth, left=0.27\textwidth of a] (A) {distributionally robust problem \eqref{problem:basic2}};
		\node[box, text width=0.24\textwidth] (B) at (A |- d) {safe approximation \eqref{problem:safe_approx_outer}};
		\node[box, text width=0.24\textwidth, below=of B] (C) {mixed-integer program \eqref{Prob: MIP_onedim}};
		\node[head, above=3mm of a] (ha) {adversarial problem\\ \footnotesize\mdseries for fixed $x^-,x^+$ and $s\in S$};
		\node[head, anchor=south] at (ha.south -| A) {outer problem\\ \footnotesize\mdseries over $(x^-,x^+,\tilde{x})\in C$};
		\draw[arr] (a) -- node[lab, right=1mm] {$\geq$, Section~\ref{sec:safe_approx}} (b);
		\draw[arr] (b) -- node[lab, right=1mm] {$=$, Lemmas\\ \ref{Thm: strong_duality}--\ref{lemma:finite_system}} (c);
		\draw[arr] (c) -- node[lab, right=1mm] {$\geq$, error\\ bounded in\\ Theorem~\ref{thm:inner_convergence}} (d);
		\draw[arr] (A) -- node[lab, right=1mm] {$\geq$, Corollary~\ref{cor:safety};\\ convergence for\\ $\delta_N\to 0$: Theorem~\ref{thm:outer_convergence}} (B);
		\draw[arr] (B) -- node[lab, right=1mm] {$=$, Theorem~\ref{Thm: MIP_onedim_safe}} (C);
		\draw[link] (A) -- node[above, font=\footnotesize] {constraint} (a);
		\draw[link] (d) -- node[above, font=\footnotesize] {$\text{val}^s_N$} (B);
	\end{tikzpicture}
	\caption{Chain of problems underlying the safe approximation. The left column shows the outer problem, the right column the adversarial problem in its constraint, for fixed $x^-,x^+$ and for each $s\in S$. Dashed arrows connect the two levels: the adversarial problems appear in the constraint of~\eqref{problem:basic2}, and the optimal values $\text{val}^s_N$ of their approximations enter the constraint of~\eqref{problem:safe_approx_outer}. A solid arrow labeled $\geq$ means that the optimal value of the lower problem underestimates that of the upper one, an arrow labeled $=$ that the optimal values coincide.}
	\label{fig:roadmap}
\end{figure}

This work is structured as follows.
In Section \ref{Sec: Problem Setting}, the distributionally robust
problem formulation including indicator functions is introduced, together with motivating examples and how the considered ambiguity sets are modeled. Section \ref{sec:safe_approx}
introduces a
semi-infinite minorant of the adversarial problem together with a suitable discretization approach that
leads to a finite-dimensional safe approximation of the original DRO problem. In addition, it is proven that the inner problem can be approximated arbitrarily well, if the discretization size tends to zero, and that this transfers to the original problem: its optimal value is approached from below, and the error bound turns the safe approximation into an a posteriori bound on the optimal value from above. Section~\ref{sec:mip} presents a mixed-integer optimization
model for the safe approximation that can be solved by standard modern software whenever the objective function is linear and the remaining feasible set is mixed-integer representable.
Computational results are
given in Section \ref{Sec:comp-results}. As motivating application, a
fundamental and difficult optimization task in material design is
studied. It is shown that particle separation under uncertainty
falls
into the class of problems studied here, for which known approaches fall
short either in modelling capabilities or in quality of the obtained
solutions.
Using realistic settings with real data, it turns out that the robust counterparts can be solved practically
efficiently via the mixed-integer linear model.

\bigskip
{\bf Our contributions} are the following:

\begin{itemize}
\item We develop a reformulation-based safe approximation of DRO constraints that contain non-convex indicator functions. The approximation uses discretization.
\item We prove that the optimal value of the inner problem of our approximation converges to the optimal value of the true inner problem for increasingly fine discretization, with an explicit error bound that does not depend on the upper-level decisions.
\item Exploiting this uniformity, we prove that the optimal value of the safe approximation converges to the optimal value of the original distributionally robust problem, provided that the latter admits strictly feasible near-optimal points, and we give a pathological example showing that without this assumption the safe approximation can be infeasible while the original problem is not. Moreover, the error bound turns the safe approximation into an a posteriori enclosure of the true optimal value.
  \item We present a safe approximation of the DRO that boils down to a MIP whenever the original upper level problem is MIP representable.
\item We show that the considered problem class appears naturally in material design in robust particle separation.
\item  Computational results on realistic particle design instances lead to two important observations: The unprotected, nominal, solution easily leads to particles that do not meet quality requirements, already for small ambiguity sets. In contrast, the robust safe approximation still satisfies these requirements, and the enclosure of Corollary \ref{cor:enclosure} certifies that its fractionation interval is at most $3.4\%$ shorter than the optimum of the original distributionally robust problem.
  \end{itemize}
\section{Problem Setting and Notation}\label{Sec: Problem Setting}
In this work, we consider optimization problems of the form
\begin{subequations}\label{problem:basic2}
	\begin{align}
		\max_{(x^-,x^+, \tilde{x}) \in C} \quad & c(x^-,x^+, \tilde{x}) \\
		\label{problem:basic_prob_constr2}\text{s.t.} \quad & \sum_{s \in S} \min_{\mathbbm{P}^s \in \u^s} a^s \mathbbm{P}^s\left([x^-, x^+]\right) \geq b,
	\end{align}
\end{subequations}
where the independent random variables $t^s\in T^s$ are distributed according to $\P^s$.

Here, $c$ denotes an arbitrary objective function, $b\in \R$ the right-hand-side, and $C \subseteq \R^k$ denotes the feasible set for the variables $x^-, x^+, \tilde{x}$, possibly constrained by additional constraints.
$S$ denotes a finite index set. The ambiguity sets $\u^s$ are sets of probability measures on a compact interval $T^s = [t^s_{\min}, t^s_{\max}] \subset \R$ with $t^s_{\min} < t^s_{\max}$ for all $s \in S$. We assume that $a^s \in \R \setminus \{0\}$ for all $s\in S$.
Constraint~\eqref{problem:basic_prob_constr2} can be considered as a distributionally robust chance constraint. Such constraints can model risk measures, see Example~\ref{example:value-at-risk}, and appear in relevant applications, for example in material science, see Section~\ref{Sec:comp-results}. Appendix~\ref{sec:notation} collects all symbols used in this article in a single table, together with the place at which each of them is introduced.

We briefly comment on the ingredients of \eqref{problem:basic2}. The set $[x^-,x^+]$ is a compact interval in $\R$, since $x^-,x^+\in \R$ and every random variable $t^s$ attains values in $T^s\subseteq \R$; the multivariate case is treated in \cite{rolfes25}. The variable $\tilde{x}$ collects all remaining decisions. It does not appear in \eqref{problem:basic_prob_constr2} and plays no role in Sections \ref{sec:safe_approx} and \ref{sec:mip}; we nevertheless keep it visible because in applications the objective function and the feasible set typically involve decisions beyond the two that determine the interval. Finally, we explicitly allow $a^s<0$. For such an $s$, the corresponding term in \eqref{problem:basic_prob_constr2} reads
\begin{equation*}
	\min_{\mathbbm{P}^s \in \u^s} a^s \mathbbm{P}^s\left([x^-, x^+]\right) = -|a^s| \max_{\mathbbm{P}^s \in \u^s} \mathbbm{P}^s\left([x^-, x^+]\right),
\end{equation*}
i.e., the adversary maximizes the probability mass that the $s$-th random variable places in the chosen interval. This is the worst case for that summand and thus in line with the distributionally robust protection sought in \eqref{Eq: DRO_constraint}; negative coefficients simply mark quantities that one wishes to keep out of the interval. In the application of Section \ref{Sec:comp-results}, we have $a^s=\left(\mathbbm{1}_{S_w}(s)-R\right)q_0(s)$, which is positive for the desired material and negative for every contaminant, see \eqref{Prob: DRO_chromatography}.

Moreover, they pose an extension to the literature with regards to the interplay between the decisions $x$ and the random variable $t$ as highlighted in the following remark.
\begin{remark}\label{remark:comparison}
	We note that $\E_{\P}\left(v(x^-,x^+,t)\right)=\sum_{s\in S} a^s \E_{\P^s}\left(v^s(x^-,x^+,t^s)\right) = \sum_{s\in S} a^s \mathbbm{P}^s\left([x^-,x^+]\right) $ encodes the expectation of a non-convex, in our case a piecewise-constant step-like function $v$ in $t\sim \mathbbm{P}$. This distinguishes our setting from the results presented in \cite{Wiesemann2014a} and \cite{Delage2010a}. Here, the underlying function $v(x,t)$ has to be both, convex and piecewise-affine in $x$ and $t$, see Condition (C3) in \cite{Wiesemann2014a} and Assumption 2 in
	\cite{Delage2010a}. In \cite{Wiesemann2014a} and \cite{Xu2017a}, there are exceptions to these assumptions, however in these exceptions the number of confidence sets has to be very low (see Remark 1 in \cite{Wiesemann2014a}), or has to be equal to $0$ (\cite{Xu2017a}).
	In the present article, we extend this
	setting by considering indicator functions $\mathbbm{1}_{[x^-,x^+]}(t)$, that generally do not satisfy any of those assumptions.
\end{remark}

\subsection{Fundamentals of DRO constraints with indicator functions}
We note that Problem~\eqref{problem:basic2} can be considered as a specification of the general DRO problem \eqref{Eq: DRO_constraint}:
If we define the function $v$ as a sum of \emph{indicator functions}, i.e.,
\begin{equation}\label{Eq: v_one-dim}
	v(x^-,x^+,t) := \sum_{s \in S} a^s \mathbbm{1}_{[x^-,x^+]}(t^s), \text{ where } \mathbbm{1}_{[x^-,x^+]}(t^s)\coloneqq \begin{cases}
		1 & \text{ if } t^s\in [x^-,x^+]\\
		0 & \text{ otherwise,}
	\end{cases}
\end{equation}
then we observe that the left-hand-side of Constraint~\eqref{problem:basic_prob_constr2} can be rewritten as
\begin{equation}
\sum_{s \in S} \min_{\mathbbm{P}^s \in \u^s} a^s \mathbbm{P}^s\left([x^-, x^+]\right) = \sum_{s \in S} \min_{\mathbbm{P}^s \in \u^s} a^s \E_{t^s \sim \P^s } \left(\mathbbm{1}_{[x^-,x^+]}(t^s)\right) = \min_{\P \in \u} \E_{t \sim \P} \left(v(x^-, x^+, t)\right)
\end{equation}
where $\u$ is the set of all probability measures on $\R^{|S|}$ whose $s$-th marginal is in $\u^s$ and independent of other marginals for all $s\in S$.
The second equality holds because the objective function is separable and the feasible set is a product: evaluating $\min_{\P\in\u}$ at the product of minimizers of the marginal problems shows ``$\geq$'', and restricting an arbitrary $\P\in \u$ to its marginals shows ``$\leq$''.

\begin{example}\label{example:value-at-risk}
The modeling power of \eqref{problem:basic2} is illustrated by one of the most prominent risk measures in risk theory - the so-called
\emph{value-at-risk}. It is often applied as a tool to aid both,
financial controlling and reporting \cite{Dowd2002a}. We refer to \cite{ghaoui2003worst} for further details on this topic. At a confidence level of $\alpha$, the value-at-risk
$\text{VaR}_{\alpha}(t)$ is defined as follows. Given a random variable $t$ with values in a compact interval $T\subset \R$ that measures the loss (e.g., of a market participant) and suppose
that this loss function is distributed by $F_t$. Truncating the domain of the loss is needed here in order to stay within the setting of \eqref{problem:basic2}, and it is no severe restriction in practice as losses are often bounded by the asset value. Then the value at risk at a confidence
level $\alpha$ is defined as
$$\text{VaR}_{\alpha}(t) \coloneqq \inf\left\{x\in \R: F_t(x) \geq \alpha\right\}.$$
Hence, by assuming that $t$ may be randomly distributed by one of the uncertain distributions $\P\in \u$, we define a \emph{robust value-at-risk}:
\begin{subequations}\label{Prob: VaR}
	\begin{align}
		\text{VaR}_{\alpha,\u}(t) \coloneqq -\max -x^+ & \\
		\st\ & \alpha \leq \min_{\P\in \u} \P\left((-\infty,x^+]\right).
	\end{align}
\end{subequations}
A safe approximation restricts the feasible set of the maximization in \eqref{Prob: VaR} and therefore underestimates its optimal value $-\text{VaR}_{\alpha,\u}(t)$. Consequently, our results provide a mixed-integer linear programming (MIP) approach to compute upper, that is, conservative, bounds on $\text{VaR}_{\alpha,\u}$, which is the direction one expects from a safe approximation of a risk measure. These bounds can be compared with the upper bounds of Theorems 7 and 8 in \cite{ghaoui2003worst}, which arise from a semidefinite relaxation under support information, whereas ours arise from a discretization of an envelope constraint.
\end{example}

The major challenge to address DRO constraints like \eqref{Eq: DRO_constraint} lies in the fact, that we need to optimize over a set of probability measures $\u$.
In case the inner minimization problem, i.e., the optimization problem of the 'adversary' is convex, e.g., of the form $b\leq \min_{a \in U} a^\top x$ with a convex set $U \subseteq \R^n$, the derivation of equivalent reformulations is well-investigated.
One then can apply standard duality-based arguments and replace the adversarial optimization problem by the feasible region of its dual. Moreover, if the considered inner problem is finite-dimensional, the resulting problem is algorithmically tractable by standard algorithms.

In our setting of non-convex indicator functions, however, such duality-based arguments do not hold in general. In order to nevertheless apply some reformulation approach, we cannot expect to derive a fully equivalent reformulation to Problem \eqref{problem:basic2}. Instead, we aim for a surrogate of the in general intractable inner (also called \emph{adversarial}) optimization problems
\begin{equation}\label{problem:inner}
	\min_{\P^s \in \u^s} a^s \E_{t^s \sim \P^s} \left(\mathbbm{1}_{[x^-,x^+]}(t^s)\right)
\end{equation}
by optimization problems over finitely many variables and constraints to enable algorithmic tractability, and with the property that the optimal value of the inner problem \eqref{problem:inner} is underestimated.

We refer to any problem that yields a valid lower bound on the optimal objective value of \eqref{problem:inner} as a \emph{minorant} of the inner problem.

Going over to a tractable minorant guarantees that solutions are feasible for the original constraint \eqref{problem:basic_prob_constr2}. It may however be that some of the robust solutions are missed.
We call a problem that in this sense replaces Problem~\eqref{problem:inner} in Problem~\eqref{problem:basic2} a \emph{safe approximation} of~\eqref{problem:basic2}. This is the notion of Definition 2.2.1 in \cite{Ben-Tal:RobustOptimization}, where a system $S$ of constraints in $x$ and auxiliary variables is called a safe approximation of a chance constraint if the $x$-component of every feasible point of $S$ is feasible for the chance constraint, and computationally tractable if the constraints of $S$ are efficiently computable; Section 2.2.1 there extends the notion to \emph{ambiguous} chance constraints, which is the situation of \eqref{problem:basic_prob_constr2}. We deviate from that definition in one respect: \cite{Ben-Tal:RobustOptimization} requires $S$ to be convex, whereas our $S$ is mixed-integer representable, see Section \ref{sec:mip}.

We note that it is not to be expected in general that a solution of a safe approximation of \eqref{problem:basic2} is close to a solution of \eqref{problem:basic2} in some appropriate metric, even if the value of the tractable minorant is arbitrarily close to the value of Problem~\eqref{problem:inner}.
Indeed, it can easily happen in many situations that replacing the adversarial problem by some minorant, for example because of better algorithmical tractability, may have the result that there is no robust feasible solution any more, while the original robust problem may indeed have robust feasible solutions. It thus does not come as a surprise that it can also happen here.
Illustrating this situation, Example~\ref{example:no_convergence_to_expect_chromatography} in Section~\ref{sec:outer} constructs an instance of \eqref{problem:basic2} for which an arbitrarily small approximation error of the optimal value of the adversarial problem renders the problem infeasible. In the same section we identify a property that this instance lacks, namely the existence of strictly feasible points of arbitrarily good objective value, and we prove that in its presence the optimal values do converge.

In this work, we will present a tractable minorant of Problem~\eqref{problem:inner} and prove that it converges to ~\eqref{problem:inner}.
While convergence of the \emph{solutions} of the safe approximation to a solution of Problem~\eqref{problem:basic2} cannot be expected in general, we show in Section~\ref{sec:outer} that the \emph{optimal values} converge under the strict feasibility condition mentioned above, and that the error bound of Theorem~\ref{thm:inner_convergence} additionally yields an a posteriori enclosure of the optimal value of Problem~\eqref{problem:basic2} that can be evaluated in practice.

After having established these key properties of the interactions between the outer and inner problem, we would briefly like to highlight that the independence of the random variables $t^s$ allows us to investigate the adversarial optimization problems for all $s\in S$ separately. Hence, we suppress $s$ whenever suitable. However, even with the assumed independence DRO constraints remain challenging as both, the inherent bilevel structure as well as the fact that the ambiguity sets $\u^s$ are generally infinite-dimensional, can cause intractability. Thus, we now specify the ambiguity sets $\u^s$ addressed in this article, which are based on the classical moment-based ambiguity sets as seen in, e.g. \cite{Delage2010a} or \cite{Wiesemann2014a}, but are linear with respect to $\P^s$ in order to allow for (linear) dualization and capture the truncated nature of the distributions considered in Section \ref{Sec:comp-results}. This choice of constraints as well as any subset of these constraints allows for the tractable safe approximation presented here.

\subsection{Ambiguity sets with moment control and envelope constraints}\label{sec:ambiguity_set}

To this end, we present two classes of constraints that can be used to define ambiguity sets. First, we are going to discuss constraints on the first two moments, where we relax the variance constraint in order to obtain a linear constraint that accounts for the influence of fluctuations in the first moment. Second, we consider an \emph{envelope constraint}, which aims to determine a compact area that bounds the density of the measure $\P^s$. The envelope is then approximated by a finite set of confidence set like constraints to achieve a finite number of constraints. Both moment constraints and confidence set constraints are standard in the literature, see e.g. \cite{Delage2010a} or \cite{Wiesemann2014a} , whereas our linear relaxation and the design of the envelope enable the tractability of the safe approximation. However, should corresponding information not be available for some application, such constraints could always be modeled in a redundant fashion.

\subsubsection{Moment constraints.}
First, we assume that the \emph{first moment}, i.e. the expectation, of $\P$ is contained in an interval (see also~\cite{Delage2010a}):

\begin{equation}\label{Eq: Sec2_first_moment}
	\EW{\P}(t) \in [\mu_-, \mu_+].
\end{equation}
Furthermore, we want to impose a restriction that excludes distributions with extensively large \emph{variance}. Note that the parameter $\sigma^2_+$ below bounds a variance and hence carries the square of the unit of $t$. As
\begin{equation}\label{Eq: Sec2_second_moment}
	\text{Var}_{\P}(t) \leq  \sigma^2_+
\end{equation}
is in general non-linear in $\P$, we cannot impose the constraint without losing the possibility to utilize linear duality to reformulate the adversarial problem.
Instead, we impose the constraint
\begin{equation}\label{Eq: Sec2_second_moment_true}
	\E_{\P}(t^2) - (\mu_-+\mu_+) \E_{\P}(t) + \mu_+\mu_- \leq  \sigma^2_+,
\end{equation}
which is linear in $\P$. Indeed, we prove next that it is fulfilled by all measures $\P$ that satisfy \eqref{Eq: Sec2_second_moment}. We additionally show that the variance over distributions that fulfill \eqref{Eq: Sec2_second_moment_true} is bounded.

\begin{lemma}
	A probability measure $\P$ that fulfills Constraints~\eqref{Eq: Sec2_first_moment} and \eqref{Eq: Sec2_second_moment} fulfills \eqref{Eq: Sec2_second_moment_true}.
\end{lemma}

\begin{proof}
	It holds that
	\begin{equation*}
		\text{Var}_\P(t) = \E_{\P}(t^2) - \E_\P(t)^2.
	\end{equation*}
	We further observe that $\E_\P(t) \in [\mu_-, \mu_+]$ implies that
	\begin{equation*}
		0 \geq \left(\E_\P(t) - \mu_-\right)\left(\E_\P(t) - \mu_+\right) = \E_\P(t)^2 - \mu_{+}\E_\P(t) - \mu_{-}\E_\P(t) + \mu_{+}\mu_{-}.
	\end{equation*}
	Hence,
	\begin{equation*}
		\text{Var}_\P(t) = \E_\P(t^2) - \E_\P(t)^2 \geq \E_{\P}(t^2) - (\mu_-+\mu_+) \E_{\P}(t) + \mu_{+}\mu_{-}.
	\end{equation*}
	As the left-hand-side of \eqref{Eq: Sec2_second_moment} is an upper bound of the left-hand-side of \eqref{Eq: Sec2_second_moment_true}, the latter is a relaxation of the former.
\end{proof}

\begin{lemma}\label{lemma:bound_variance_relaxation}
	The variance of a probability measure $\P$ that fulfills Constraints~\eqref{Eq: Sec2_first_moment} and \eqref{Eq: Sec2_second_moment_true} is not larger than $\sigma^2_+ + \left(\frac{\mu_-+\mu_+}{2}\right)^2 - \mu_{+}\mu_{-}$.
\end{lemma}

\begin{proof}
	Using that the variance is the smallest second moment around any point, i.e., $\text{Var}_\P(t)=\min_{c\in \R}\E_\P\left((t-c)^2\right)$, evaluated at $c=\frac{\mu_-+\mu_+}{2}$, it holds that
	\begin{align*}
		\text{Var}_\P(t) &\leq \E_\P\left(\left(t - \frac{\mu_-+\mu_+}{2}\right)^2\right)\\
		&= \E_{\P}(t^2) - (\mu_-+\mu_+) \E_{\P}(t) + \mu_+\mu_- + \left(\frac{\mu_-+\mu_+}{2}\right)^2 - \mu_+\mu_-\\
		&\leq \sigma^2_+ + \left(\frac{\mu_-+\mu_+}{2}\right)^2 - \mu_+\mu_-.
	\end{align*}
\end{proof}
Hence, we have demonstrated that the constraint \eqref{Eq: Sec2_second_moment_true}, which is linear in $\P$ is a relaxation of \eqref{Eq: Sec2_second_moment}, which still provides a bounded variance. Moreover, for $\mu_-=\mu_+=\E_{\P}(t)$, Lemma \ref{lemma:bound_variance_relaxation} shows that the relaxation is sharp.

\subsubsection{Envelope constraints.}
In order to model ambiguity sets that are not unnecessarily large, we allow further restrictions on the uncertain measure $\P$ as follows: Assuming that $\P$ has a probability density, denoted as $\rho$, we define an \emph{envelope constraint} as
\begin{equation}\label{Eq: envelope}
	\uRTD{}{t} \leq \bound{}{t} \text{ for all }t\in T,
\end{equation}
with $\bound{}{t}\colon T \to \R_{\geq 0}$ bounded and integrable. Whenever we use \eqref{Eq: envelope}, it is understood as part of the definition of the ambiguity set that $\P$ is absolutely continuous with respect to the Lebesgue measure with density $\rho$; measures without a density are thus excluded. For notational convenience we extend $\bound{}{\cdot}$ by zero outside of $T$.
The envelope function restricts the value of the density of $\P$. Thus, it incorporates additional information into the ambiguity set.
It excludes measures that concentrate all the probability mass around a single point and thus have a high density at this point, e.g., Dirac point measures.
If the uncertainty is parametrized, bounds for \eqref{Eq: envelope} may be fairly simple to obtain. This can be illustrated by the following example:

\begin{example}\label{example:envelope_normal}
	Let $T=[t_{\min},t_{\max}]$ and $\uRTD{\mu}{t}$ be the probability density function of a normal distribution with expectation $\mu$, restricted to $T$. Suppose $\mu$ varies between $\mu_{-}$ and $\mu_{+}$ and $\uRTD{}{t}$ lies in the convex hull of $\left\{\uRTD{\mu}{t}: \mu \in [\mu_{-},\mu_{+}]\right\}$. Then, a valid definition of $\bound{}{t}$ would be
	\[
	\bound{}{t} \coloneqq\left\{\begin{array}{lll} \uRTD{\mu_{-}}{t} , & t\in [t_{\min}, \mu_{-}] \\
		\uRTD{\mu }{\mu} , & t\in [\mu_{-}, \mu_{+}]\\
		\uRTD{\mu_{+}}{t} , & t\in [\mu_{+},t_{\max}].\end{array}\right.
	\]
\end{example}
In addition, envelope constraints can in some instances effectively replace discrepancy based constraints such as Wasserstein balls.
\begin{remark}
	In discrepancy-based DRO models,
	$\u$ is also constrained, e.g. via Wasserstein balls around some distribution. Given a nominal
	probability distribution $\P_0\in \u$, these measures usually
	limit the probability mass that needs to be transferred to reach
	another $\P\in \u$. In our computational results in Section \ref{Sec:comp-results}, we use the envelope constraints to achieve a similar effect. Furthermore, if we consider
	an ambiguity set $\u$, that contains realistic uncertainties with a rather large Wasserstein distance between each other, an additional Wasserstein constraint would not strengthen the model further. This is illustrated by the application in Section \ref{Sec:comp-results}. Finally, the computational results show that with the ambiguity sets considered, the robust solutions of a relevant application satisfy the imposed quality requirement, with a certified bound on their distance to the optimum of the underlying distributionally robust problem.
\end{remark}

Collecting the above, the ambiguity sets considered in this article are
\begin{equation}\label{Eq: ambiguity_set}
	\u^s = \left\{ \P \ll \Lcal \ \middle| \
	\begin{aligned}
		&\P(T^s)=1,\quad \rho_{\P}(t)\leq \bar{\rho}^s(t) \text{ for almost all } t \in T^s,\\
		&\E_{\P}(t)\in [\mu^s_-,\mu^s_+],\quad \E_{\P}(t^2)-(\mu^s_-+\mu^s_+)\E_{\P}(t)+\mu^s_+\mu^s_-\leq \sigma^{2,s}_+
	\end{aligned}
	\right\},
\end{equation}
where $\Lcal$ denotes the Lebesgue measure on $T^s$ and $\rho_\P$ the density of $\P$. Any subset of these constraints may be used as well; constraints for which no information is available can be modeled redundantly. Since every $\P\in\u^s$ has a density bounded by $\bar{\rho}^s$, the ambiguity sets contain no atoms, so that worst-case measures of~\eqref{problem:inner} are not discrete.

\section{Safe Approximation of a DRO Constraint}\label{sec:safe_approx}
In this section, we construct a minorant of the adversarial problem \eqref{problem:inner}.
Replacing the adversarial problems in \eqref{problem:basic2} by it yields an algorithmically tractable safe approximation of Problem~\eqref{problem:basic2}.
Section~\ref{sec:inner} constructs this minorant and bounds its error, Section~\ref{sec:outer} transfers the bound to Problem~\eqref{problem:basic2} itself.

\subsection{Approximation of the adversarial problem}\label{sec:inner}

We carry out the chain of Figure \ref{fig:roadmap} in the following order. Dualizing \eqref{problem:inner} faces two obstacles: the integrand is discontinuous, and the envelope constraint is semi-infinite since it is imposed at every point of $T$. Underestimating the former by a continuous function and aggregating the latter over the cells of a grid of width $\delta_N$ yields the minorant \eqref{Prob: Primal_Purity_Constraint_envelope}. This minorant is dualized without a duality gap (Lemma \ref{Thm: strong_duality}). Its dual still contains a semi-infinite constraint, but the infimum of its left-hand side can be computed in closed form (Lemma \ref{lemma:inf_f}). Thus, the semi-infinite system can be replaced by an equivalent finite one (Lemma \ref{lemma:finite_system}). That system is not yet linear, but restricting it to a sufficient system removes this drawback (Lemma \ref{lemma:linear_system}) at the price of a second underestimation. Re-dualizing the restricted problem gives \eqref{problem:safe_approx_primal}, a linear program over measures supported on the grid. Theorem \ref{thm:inner_convergence} then bounds what the two underestimations together cost: the error is bounded linearly in $\delta_N$ and, in particular, does not depend on $x^-$ and $x^+$, which aids in bounding the outer problem in Section \ref{sec:outer}.

Let $\Ccal(T)$ be the set of continuous functions on $T$, $\Mcal(T)$ be the set of \emph{signed Radon measures}, and $\Mcal(T)_{\geq 0}$ the set of \emph{non-negative Radon measures}. As illustrated in Section III.3.2 in \cite{Barvinok2002a}, the corresponding inner product $\langle \cdot , \cdot \rangle: \Ccal(T)\times \Mcal(T) \rightarrow \R$ defined by
$$\langle g, \P\rangle \coloneqq \int_T g d\P$$
constitutes a non-degenerate inner product or a \emph{duality}. As probability measures are specific signed Radon measures, the above duality is well-defined. We aim for underestimating the objective function and relax the constraints of \eqref{problem:inner} with the help of the duality, i.e., we aim for continuous functions $g, g_i, i\in I$ with $I$ being some index set, and numbers $d_i$ for $i\in I$ that enable us to express a minorant of Problem~\eqref{problem:inner} as
\begin{subequations}\label{Prob: DRO_with_undefined_ambiguity_set}
\begin{align}
	\min_{\P \in \Mcal(T)_{\geq 0}}~&  \left\langle g,\P\right\rangle & \\
	\label{Prob: DRO_with_undefined_ambiguity_set_constr}\text{s.t. }& \left\langle g_i,\P\right\rangle \geq d_i & \text{ for all } i \in I.
\end{align}
\end{subequations}
To express that $\P$ is a probability measure with mass $1$, we impose the constraint
\begin{equation*}
	\langle 1,\P\rangle \geq 1 \text{ and }\langle -1,\P\rangle \geq -1
\end{equation*}
Constraint \eqref{Eq: Sec2_first_moment} is reformulated to
\begin{equation*}
	\langle t,\P\rangle \geq \mu_- \text{ and }\langle -t,\P\rangle \geq -\mu_+.
\end{equation*}
Constraint \eqref{Eq: Sec2_second_moment_true} can be expressed with the duality by
\begin{equation}\label{cons:variance_conservative}
	\langle t^2-(\mu_-+\mu_+) t,\P\rangle \leq \sigma^2_+ - \mu_{+}\mu_{-}.
\end{equation}

We approximate Constraints \eqref{Eq: envelope} with a finite number of constraints of the form \eqref{Prob: DRO_with_undefined_ambiguity_set_constr}, that resemble \emph{confidence set constraints}, i.e., constraints that restrict the probability mass of an event $D_i$ by $\P(t\in D_i)\geq d_i$.
To this end, we consider sequences $K_N \in \mathbbm{N}$ with $K_N \to \infty$, the grid widths $\delta_N \coloneqq \frac{t_{\max}-t_{\min}}{K_N}$, sequences $\eta_N > 0$ with $\frac{\eta_N}{\delta_N} \to 0$, and the finite set $T_N\coloneqq \{t_{\min} + k\delta_N \mid k = 0,\ldots,K_N\}$. In particular, $\min T_N = t_{\min}$ and $\max T_N = t_{\max}$, and the cells $[\tau,\tau+\delta_N)$, $\tau \in T_N\setminus\{t_{\max}\}$, partition $[t_{\min}, t_{\max})$. Since $\bar{\rho}$ is extended by zero, $T$ may be replaced by any interval that contains the support of $\bar{\rho}$.
We further fix, for $\eta > 0$ and $\tau < \tilde{\tau} \in T$, an arbitrary continuous function $\mathbbm{1}^{\eta}_{[\tau, \tilde{\tau}]}$ that approximates $\mathbbm{1}_{[\tau, \tilde{\tau}]}$ in the sense that

\begin{equation*}
	\mathbbm{1}^{\eta}_{[\tau, \tilde{\tau}]}(t) := \begin{cases}
		1 \text{ if } t \in [\tau, \tilde{\tau}]\\
		0 \text{ if } t \notin [\tau - \eta, \tilde{\tau} + \eta]\\
		\in [0, 1] \text{ otherwise.}
	\end{cases}
\end{equation*}
Such functions exist; the piecewise linear interpolation on $[\tau-\eta,\tau]$ and on $[\tilde{\tau},\tilde{\tau}+\eta]$ is one example. Every choice is admissible, since all statements below use only the pointwise estimates $\mathbbm{1}_{[\tau, \tilde{\tau}]}\leq \mathbbm{1}^{\eta}_{[\tau, \tilde{\tau}]}\leq \mathbbm{1}_{[\tau-\eta, \tilde{\tau}+\eta]}$ and the continuity, which is required because the duality of Problem \eqref{Prob: DRO_with_undefined_ambiguity_set} pairs measures with continuous functions.
If \eqref{Eq: envelope} holds for $\P$, then we can derive that
\begin{equation*}
	\langle \mathbbm{1}^{\eta_N}_{[\tau,\tau+\delta_N]}, \P \rangle = \int_{\tau - \eta_N}^{\tau + \delta_N + \eta_N} \mathbbm{1}^{\eta_N}_{[\tau,\tau+\delta_N]} d\P \leq \int_{\tau - \eta_N}^{\tau + \delta_N + \eta_N} \mathbbm{1}^{\eta_N}_{[\tau,\tau+\delta_N]} \bar{\rho}(t) dt%
\end{equation*}
Hence, replacing \eqref{Eq: envelope} by
\begin{equation}\label{Eq: envelope_discretized}
	\langle \mathbbm{1}^{\eta_N}_{[\tau,\tau+\delta_N]}, \P \rangle \leq \bar{\rho}^+_N(\tau) \qquad \forall \tau \in T_N
\end{equation}
with
\begin{equation*}
	\bar{\rho}^+_N(\tau) \coloneqq \int_{\tau - \eta_N}^{\tau + \delta_N + \eta_N} \mathbbm{1}^{\eta_N}_{[\tau,\tau+\delta_N]} \bar{\rho}(t) dt
\end{equation*}
leads to a relaxation of the semi-infinite Problem \eqref{problem:inner} with finitely many constraints.

To underestimate the objective function of \eqref{problem:inner} with the help of the duality, we underestimate $a\mathbbm{1}_{[x^-, x^+]}$ by $a \mathbbm{1}^c_{[x^-, x^+]}$, where $\mathbbm{1}^c_{[x^-, x^+]}$ again denotes an arbitrary continuous function subject to
\begin{equation*}
	\mathbbm{1}^c_{[x^-, x^+]}(t) :=
	\begin{cases}
		0, \text{ if }t \notin [x^-, x^+],\\
		1, \text{ if }t \in [\tau^-_N, \tau^+_N],\\
		\in [0, 1] \text{ otherwise,}
	\end{cases}
\end{equation*}
if $a > 0$, with
\begin{equation}\label{taudef}
\tau^-_N := \min\left\{\tau \in t_{\min}+\delta_N\mathbbm{Z}\vert \tau > x^-\right\}, \quad \tau^+_N := \max\left\{\tau \in t_{\min}+\delta_N\mathbbm{Z}\vert \tau < x^+\right\},
\end{equation}
and
\begin{equation*}
	\mathbbm{1}^c_{[x^-, x^+]}(t) :=
	\begin{cases}
		0, \text{ if }t \notin (\tau^-_N, \tau^+_N),\\
		1, \text{ if }t \in [x^-, x^+],\\
		\in [0, 1] \text{ otherwise,}
	\end{cases}
\end{equation*}
if $a < 0$, with
\begin{equation}\label{taudef2}
\tau^-_N := \max\left\{\tau \in t_{\min}+\delta_N\mathbbm{Z}\vert \tau < x^-\right\}, \quad \tau^+_N := \min\left\{\tau \in t_{\min}+\delta_N\mathbbm{Z}\vert \tau > x^+\right\}.
\end{equation}

The minimum and the maximum in \eqref{taudef} and \eqref{taudef2} are taken over the unbounded grid $t_{\min}+\delta_N\mathbbm{Z}$ rather than over $T_N = (t_{\min}+\delta_N\mathbbm{Z})\cap T$, so that $\tau^-_N$ and $\tau^+_N$ exist for all $x^-,x^+\in T$, at the price of allowing them to lie outside $T$. For $a<0$ this yields $\tau^-_N < x^- \leq x^+ < \tau^+_N$, so that the two never cross. For $a>0$ they do cross, i.e.\ $\tau^-_N > \tau^+_N$, exactly when no point of $t_{\min}+\delta_N\mathbbm{Z}$ lies strictly between $x^-$ and $x^+$, which requires $x^+ - x^- \leq \delta_N$. In that case $[\tau^-_N, \tau^+_N]$ is empty, the second line of the definition imposes no condition, and $\mathbbm{1}^c_{[x^-, x^+]}$ may be chosen as the zero function; the estimate below remains valid, albeit with the weakest possible bound.
We derive that
\begin{equation*}
	a\P\left([x^-, x^+]\right) = \int_T a\mathbbm{1}_{[x^-, x^+]}(t) d\P(t) \geq \int_T a\mathbbm{1}^c_{[x^-, x^+]}(t) d\P(t) = \langle a\mathbbm{1}^c_{[x^-,x^+]},\P\rangle.
\end{equation*}
Combining the observations, we obtain a minorant of \eqref{problem:inner} in the form of \eqref{Prob: DRO_with_undefined_ambiguity_set} as:
\begin{subequations}
	\label{Prob: Primal_Purity_Constraint_envelope}
	\begin{align}
		\inf_{\P \in \mathcal{M}(T)_{\ge 0}}~& \langle a\mathbbm{1}^c_{[x^-,x^+]},\P\rangle && \label{Constr: Objective_Primal_Purity_Constraint}\\
		\text{s.t.}~&\langle 1, \P \rangle \geq 1 \label{Constr: Primal_Purity_Constraint_envelope1}\\
		&\langle -1, \P \rangle \geq -1 \\
		& \langle -t, \P\rangle \geq -\mu_{+},\label{Constr: First_Moment1} \\
		& \langle t, \P\rangle \ge \mu_{-}, \label{Constr: First_Moment2}\\
		&\langle -t^2+(\mu_-+\mu_+) t ,\P\rangle \geq - \sigma^2_+ + \mu_+\mu_-  \label{Constr: Second_Moment}\\
		& \langle -\mathbbm{1}^{\eta_N}_{[\tau,\tau+\delta_N]},\P \rangle \geq - \bar{\rho}^+_N(\tau) && \text{ for all } \tau \in T_N. \label{Constr: envelope_discretized}
	\end{align}
\end{subequations}

In order to dualize Problem~\eqref{Prob: Primal_Purity_Constraint_envelope}, we need to determine an adjoint operator to the primal constraint operator
\begin{equation*}
\mathcal{A} : \Mcal(T) \rightarrow \R^{5 + |T_N|},
\end{equation*}
\begin{equation*}
	\P \mapsto \left(\langle 1, \P \rangle, \langle -1, \P \rangle, \langle -t, \P \rangle, \langle t, \P \rangle, \langle -t^2 + (\mu_-+\mu_+) t, \P \rangle, \langle -\mathbbm{1}^{\eta_N}_{[\tau, \tau + \delta_N]}, \P \rangle_{\tau \in T_N}\right)^T.
\end{equation*}
To this end, let us denote the dual variables originating from constraints \eqref{Constr: Primal_Purity_Constraint_envelope1} -- \eqref{Constr: Second_Moment} by $y_k$, e.g., for \eqref{Constr: First_Moment1}, we have $k=3$. Additionally, we denote the dual variables that correspond to the aggregated envelope constraint \eqref{Constr: envelope_discretized} as $z \in \mathbb{R}_{\ge 0}^{T_N}$.
The adjoint operator with respect to $\mathcal{A}$ is
\begin{equation*}
	\mathcal{B}: \R^{5 + |T_N|}\rightarrow \Ccal(T),
\end{equation*}
\begin{equation*}
	\left(y_1,\ldots , y_5, (z_\tau)_{\tau \in T_N}\right)\mapsto y_1-y_2-y_3t+y_4t + y_5\left(-t^2+ (\mu_-+\mu_+) t\right) + \sum_{\tau\in T_N} \mathbbm{1}_{[\tau,\tau+\delta_N]}^{\eta_N}(t) z_\tau.
\end{equation*}
$\mathcal{B}$ is unique due to Riesz' representation theorem, see e.g., \cite{Brezis2010a}.
Consequently, as a dual program to Problem~\eqref{Prob: Primal_Purity_Constraint_envelope} we obtain:
\begin{subequations}
	\label{Prob: Dual_Purity_Constraint}
	\begin{align}
		\sup_{y \in \mathbb{R}^{5}_{\ge 0}, z \in \mathbb{R}_{\ge 0}^{T_N}}& \langle (1,-1, -\mu_{+}, \mu_{-}, - \sigma^2_+ +\mu_+\mu_-),y\rangle - \sum_{\tau \in T_N} \bar{\rho}^+_N(\tau) \boundvariable{\tau} \label{Constr: dual_objective}\\
		\st\ & a\mathbbm{1}_{[x^-,x^+]}^c(t) - y_{1} + y_{2}+y_{3} t - y_{4} t + y_{5} \left( t^2-(\mu_-+\mu_+) t \right) \notag\\
		& \qquad\qquad\qquad\qquad\qquad + \sum_{\tau \in T_N} \mathbbm{1}^{\eta_N}_{[\tau,\tau+\delta_N]}(t) \boundvariable{\tau}\in \Ccal(T)_{\ge 0}\label{Constr: dual_purity}
	\end{align}
\end{subequations}
As typical in reformulation
approaches in robust optimization, we aim at using strong duality. Indeed, next we establish strong duality between \eqref{Prob:
	Primal_Purity_Constraint_envelope} and \eqref{Prob:
	Dual_Purity_Constraint} that can be seen as a direct
consequence of Corollary 3.1 in \cite{Shapiro2000a} or as a direct
consequence of the dualization theory illustrated, e.g., in \cite{Barvinok2002a}.
To apply it, we assume that the feasible set of Problem~\eqref{problem:inner} is non-empty, e.g., because a nominal distribution is contained.

\begin{lemma}\label{Thm: strong_duality}
	Suppose that the feasible set of Problem~\eqref{Prob: Primal_Purity_Constraint_envelope} is nonempty. Then, the duality gap of the problems \eqref{Prob: Primal_Purity_Constraint_envelope} and \eqref{Prob: Dual_Purity_Constraint} is zero.
\end{lemma}

\begin{proof}
	We observe that Problem~\eqref{Prob: Primal_Purity_Constraint_envelope} is "consistent" as defined in \cite{Shapiro2000a}, since its feasible set is not empty. Furthermore, $T$ is compact and the functions in the objective as well as in the constraints of \eqref{Prob: Primal_Purity_Constraint_envelope} are continuous. Hence, strong duality follows from Corollary 3.1 in \cite{Shapiro2000a}.
\end{proof}

First, we can neglect to explicitly demand continuity in \eqref{Constr: dual_purity} since the left hand side already consists only of continuous functions. Hence, the above program is a semi-infinite program, particularly a linear program with infinitely many constraints.

Second, in order to simplify Constraint \eqref{Constr: dual_purity} one observes that we can guarantee \eqref{Constr: dual_purity} by going back to indicator functions (instead of their continuous approximations), as \eqref{Constr: dual_purity} is implied by
\begin{align}\label{Constr: dual_purity3}
	& a\mathbbm{1}_{I(a, x^-, x^+)}(t) - y_{1} + y_{2}+y_{3} t - y_{4} t + y_{5} \left(t^2-(\mu_-+\mu_+) t\right)\notag\\
	& \qquad + \sum_{\tau \in T_N} \mathbbm{1}_{[\tau,\tau+\delta_N)}(t) \boundvariable{\tau} \geq 0~ \text{ for all } t \in T.
\end{align}
with $I(a, x^-, x^+) := [\tau^-_N, \tau^+_N)$.
This is due to the fact that given $\tau^-_N, \tau^+_N$ have been defined for $a>0$ in \eqref{taudef} and for $a<0$ in \eqref{taudef2}, we have $a\mathbbm{1}_{[x^-,x^+]}^c \geq a \mathbbm{1}_{I_{(a,x^-,x^+)}}$ at every $t\in T$: for $a>0$ because $\mathbbm{1}^c_{[x^-,x^+]}$ equals $1$ on $[\tau^-_N,\tau^+_N] \supseteq I(a,x^-,x^+)$ and is non-negative elsewhere, and for $a<0$ because $\mathbbm{1}^c_{[x^-,x^+]}$ is at most $1$ on $I(a,x^-,x^+)$ and vanishes outside $(\tau^-_N,\tau^+_N)$, in particular at $\tau^+_N$ itself, where $I(a,x^-,x^+)$ has already ended. Together with the observation that $\mathbbm{1}_{[\tau,\tau+\delta_N]}^{\eta_N} \geq \mathbbm{1}_{[\tau,\tau+\delta_N)}$, this indeed illustrates that \eqref{Constr: dual_purity3} yields an inner approximation of \eqref{Constr: dual_purity}. Consequently, solutions $x^-,x^+$ feasible for \eqref{Constr: dual_purity3} are feasible for \eqref{problem:inner} as well.

The left-hand side of Constraint~\eqref{Constr: dual_purity3} is the sum of indicator functions of intervals and the polynomial
$$p_y(t)\coloneqq - y_{1} + y_{2}+y_{3} t - y_{4} t + y_{5} \left(t^2-(\mu_-+\mu_+) t\right)$$
where $t_y \in \argmin_{t\in T} p_y(t)$ denotes a \emph{minimizer} of $p_y$ over $T$. If $y_5>0$ and the vertex of the parabola $p_y$ lies in $T$, then $t_y$ is this vertex; otherwise $p_y$ is monotone on $T$ and $t_y$ is a boundary point of $T$. We introduce
$$f_{y,z}(t) \coloneqq p_y(t) + \sum_{\tau\in T_N} \mathbbm{1}_{[\tau,\tau+\delta_N)}(t)z_\tau + a\mathbbm{1}_{I(a, x^-, x^+)}(t)$$
for the left-hand side of Constraint~\eqref{Constr: dual_purity3}.
We can prove the following Lemma about the infimum of $f_{y, z}$.

\begin{lemma}\label{lemma:inf_f}
	It holds for given $y, z \geq 0$ and $x^-, x^+ \in T$ that
	\begin{equation*}
		\inf_{t\in T} f_{y, z}(t) = \min\left\{ f_{y, z}(t_y), \min_{\tau \in T_N} f_{y, z}(\tau), \min_{\tau \in T_N \setminus \{t_{\max}\}} p_y(\tau + \delta_N) + z_{\tau} + a \mathbbm{1}_{I(a, x^-, x^+)}(\tau) \right\}.
	\end{equation*}
\end{lemma}

\begin{proof}
	We note that for given $y, z$ the function $f_{y,z}$ is piecewise continuous on the intervals $[\tau, \tau + \delta_N)$ for $\tau \in T_N$.
	If $y_5 = 0$, or if the vertex of $p_y$ lies outside $T$, then $p_y$ and hence $f_{y,z}$ is monotone on every such interval and the claim holds.
	Otherwise, for all but the interval containing $t_y$, the function is monotonous.
	Hence, the infimum of $f_{y, z}$ is either the function value at the global minimum $t_y$ of $p_y$ or one of its limits at the boundaries of $[\tau, \tau + \delta_N)$ for a $\tau \in T_N$. We calculate that
	\begin{equation*}
		\lim_{t \to \tau, t > \tau} f_{y, z}(t) = p_y(\tau) + z_\tau + a\mathbbm{1}_{I(a, x^-, x^+)}(\tau) = f_{y, z}(\tau),
	\end{equation*}
	as $f_{y, z}$ is a right-continuous function, and
	\begin{equation*}
		\lim_{t \to \tau + \delta_N, t < \tau + \delta_N} f_{y, z}(t) = p_y(\tau + \delta_N) + z_\tau + a\mathbbm{1}_{I(a, x^-, x^+)}(\tau),
	\end{equation*}
	as $p_y$ is continuous, $\mathbbm{1}_{[\tau, \tau+\delta_N)}(t) = \mathbbm{1}_{[\tau, \tau+\delta_N)}(\tau)$, and $\mathbbm{1}_{I(a, x^-, x^+)}(t) = \mathbbm{1}_{I(a, x^-, x^+)}(\tau)$ for all $t \in [\tau, \tau+\delta_N)$.
	The claim follows.
\end{proof}

With the help of Lemma \ref{lemma:inf_f} we can replace the (infinite) System \eqref{Constr: dual_purity3} by a finite system.

\begin{lemma}\label{lemma:finite_system}
	Constraint \eqref{Constr: dual_purity3} is equivalent to the system
	\begin{subequations}\label{constr:finite}
		\begin{align}
			\label{constr:finite1}f_{y, z}(t_y) \geq 0 & \\
			\label{constr:finite2}f_{y, z}(\tau) \geq 0 & \text{ for all } \tau \in T_N\\
			\label{constr:finite3}p_y(\tau + \delta_N) + z_\tau + a\mathbbm{1}_{I(a, x^-, x^+)}(\tau) \geq 0 & \text{ for all } \tau \in T_N \setminus \{t_{\max}\}.
		\end{align}
	\end{subequations}
\end{lemma}

\begin{proof}
	This is an immediate consequence of Lemma \ref{lemma:inf_f}: Constraint \eqref{Constr: dual_purity3} is equivalent to
	\begin{equation*}
		\inf_{t\in T} f_{y, z}(t) \geq 0,
	\end{equation*}
	which is, according to Lemma \ref{lemma:inf_f}, equivalent to
	\begin{equation*}
		\min\left\{ f_{y, z}(t_y), \inf_{\tau \in T_N} f_{y, z}(\tau), \inf_{\tau \in T_N \setminus \{t_{\max}\}} p_y(\tau + \delta_N) + z_{\tau} + a \mathbbm{1}_{I(a, x^-, x^+)}(\tau) \right\} \geq 0,
	\end{equation*}
	which is equivalent to Constraints \eqref{constr:finite}. The claim follows.
\end{proof}

We note that Constraints~\eqref{constr:finite2} and \eqref{constr:finite3} are, for fixed $x^-, x^+$, linear. Constraint~\eqref{constr:finite1} is not linear as it contains $t_y$, which itself depends on the choice of the coefficients $y$. Hence, we aim for linear constraints that ensure \eqref{constr:finite1} in the next lemma.

\begin{lemma}\label{lemma:linear_system}
	Constraints~\eqref{constr:finite1}-\eqref{constr:finite3} are implied by the constraints
	\begin{subequations}
		\begin{align}
			\label{constr:finite2.2}f_{y, z}(\tau) \geq \frac{\delta_N^2}{4}y_5 & \text{ for all } \tau \in T_N\\
			\label{constr:finite3.2}p_y(\tau + \delta_N) + z_\tau + a\mathbbm{1}_{I(a, x^-, x^+)}(\tau) \geq \frac{\delta_N^2}{4}y_5 & \text{ for all } \tau \in T_N \setminus \{t_{\max}\}.
		\end{align}
	\end{subequations}
\end{lemma}

\begin{proof}
	We first treat the degenerate cases. If $y_5 = 0$, then $p_y$ is affine, and if $y_5>0$ and the vertex of $p_y$ lies outside $T$, then $p_y$ is monotone on $T$. In both cases $f_{y,z}$ is monotone on every interval $[\tau,\tau+\delta_N)$, so that by the argument in the proof of Lemma \ref{lemma:inf_f} the infimum of $f_{y,z}$ is the smallest of the values and one-sided limits occurring in \eqref{constr:finite2} and \eqref{constr:finite3}. Hence \eqref{constr:finite1} is implied by \eqref{constr:finite2} and \eqref{constr:finite3}, which in turn are implied by \eqref{constr:finite2.2} and \eqref{constr:finite3.2} because $y_5\geq 0$. For the remainder of the proof we may thus assume that $y_5>0$ and that $t_y$ is the vertex of $p_y$.
	We define
	\begin{equation*}
		\bar{\tau} := \sup\left\{\tau \in T_N \vert \tau \leq t_y\right\},
	\end{equation*}
	which is well defined since $\min T_N = t_{\min} \leq t_y$.
	We note that either $t_y - \bar{\tau} \leq \frac{\delta_N}{2}$, or $\bar{\tau} < t_{\max}$ and $\bar{\tau} + \delta_N - t_y \leq \frac{\delta_N}{2}$, since $\bar{\tau} = t_{\max}$ implies $t_y = t_{\max} = \bar{\tau}$. We further note that representing $p_y$ in its vertex form yields $p_y(t) = y_5(t - t_y)^2 + p_y(t_y)$.
	In the first case, we calculate that
	\begin{equation*}
		f_{y, z}(\bar{\tau}) - f_{y, z}(t_y) = p_y(\bar{\tau}) - p_y(t_y) = y_5(\bar{\tau} - t_y)^2 + p_y(t_y) - p_y(t_y) \leq \frac{\delta_N^2}{4}y_5.
	\end{equation*}
	It follows that
	\begin{equation*}
		f_{y, z}(\bar{\tau}) - \frac{\delta_N^2}{4}y_5 \leq f_{y, z}(t_y).
	\end{equation*}
	Constraint~\eqref{constr:finite2.2} implies \eqref{constr:finite1}, as
	\begin{equation*}
		f_{y, z}(\tau) \geq \frac{\delta_N^2}{4}y_5 \text{ for all }\tau \in T_N \implies f_{y, z}(\bar{\tau}) - \frac{\delta_N^2}{4}y_5 \geq 0 \implies f_{y, z}(t_y) \geq 0.
	\end{equation*}
	The second case is analogous via \eqref{constr:finite3.2} and since \eqref{constr:finite2} and \eqref{constr:finite3} are immediately implied by \eqref{constr:finite2.2} and \eqref{constr:finite3.2} the claim follows.
\end{proof}

We conclude that the finitely many linear Constraints~\eqref{constr:finite2.2} and \eqref{constr:finite3.2} are an inner approximation of the semi-infinite Constraint~\eqref{Constr: dual_purity3}. This means that these new constraints ensure that \eqref{Constr: dual_purity3} is satisfied, but that \eqref{Constr: dual_purity3} may have more feasible solutions.

In order to ease the representation of the model by a mixed-integer linear problem, see Section~\ref{sec:mip}, we further approximate $\max \eqref{Constr: dual_objective} \text{ s.t. }\eqref{constr:finite2.2}, \eqref{constr:finite3.2}$ from below. To this end, we replace the function $a\mathbbm{1}_{I(a, x^-, x^+)}$ in \eqref{constr:finite2.2} and \eqref{constr:finite3.2} by $a\mathbbm{1}_{\tilde{I}(a, x^-, x^+)}$, with
\begin{equation*}
	\tilde{I}(a, x^-, x^+) := [x^- + \delta_N, x^+ - 2 \delta_N] \text{ if }a > 0,
\end{equation*}
and
\begin{equation*}
	\tilde{I}(a, x^-, x^+) := (x^- - 2\delta_N, x^+ + \delta_N) \text{ if }a < 0.
\end{equation*}
Since $[x^-,x^+]\supseteq I_{(a,x^-,x^+)}=[\tau_N^-,\tau_N^+)\supseteq \tilde{I}(a, x^-, x^+)$ for $a>0$ and $[x^-,x^+]\subseteq I_{(a,x^-,x^+)}=[\tau_N^-,\tau_N^+)\subseteq \tilde{I}(a, x^-, x^+)$ for $a<0$, we obtain $a\mathbbm{1}_{I(a, x^-, x^+)}\geq a\mathbbm{1}_{\tilde{I}(a, x^-, x^+)}$, so that the modified constraints are stronger and the optimal value can only decrease.

In addition, we note that the resulting problem has an interesting interpretation, that becomes obvious if we re-dualize it to the primal variable space. To this end, let us denote the dual variable to the modified Constraint \eqref{constr:finite2.2} and \eqref{constr:finite3.2} by $w^-$ and $w^+$ respectively. Then, the dual of the problem is
\begin{subequations}\label{problem:safe_approx_primal}
	\begin{align}
		\min_{\substack{w^- \in \R^{T_N}_{\geq 0},\\ \quad\ \ w^+ \in \R^{T_N\setminus\{t_{\max}\}}_{\geq 0}}}~&  a \sum_{\tau \in T_N} \mathbbm{1}_{\tilde{I}(a, x^-, x^+)}(\tau) w_\tau^- + a \sum_{\tau \in T_N} \mathbbm{1}_{\tilde{I}(a, x^-, x^+)}(\tau) w_\tau^+ \\
		\text{s.t.}~&\sum_{\tau \in T_N} w_\tau^- + w_\tau^+ = 1\label{constr:prob_measure}\\
		\label{problem:safe_approx_primal_exp}& \sum_{\tau \in T_N} \tau w_\tau^- + (\tau + \delta_N)w_\tau^+ \in [\mu_-, \mu_+]\\
		\label{constr:variance}&\sum_{\tau \in T_N} \left(\tau^2 - (\mu_-+\mu_+)\tau - \frac{\delta_N^2}{4}\right) w_\tau^- \notag\\
		& \qquad + \left((\tau + \delta_N)^2 - (\mu_-+\mu_+)(\tau + \delta_N) - \frac{\delta_N^2}{4}\right) w_\tau^+ \notag\\
		& \qquad \leq \sigma^2_+ - \mu_+\mu_-\\
		& w_\tau^- + w_\tau^+ \leq \bar{\rho}^+_N(\tau) \text{ for all } \tau \in T_N.
	\end{align}
\end{subequations}
Exploiting \eqref{constr:prob_measure}, we can reformulate Constraint~\eqref{constr:variance} to
\begin{equation}\label{constr:variance_ref}
	\sum_{\tau \in T_N} \left(\tau^2 - (\mu_-+\mu_+)\tau\right) w_\tau^- + \left((\tau + \delta_N)^2 - (\mu_-+\mu_+)(\tau + \delta_N)\right) w_\tau^+ \leq \sigma^2_+ - \mu_+\mu_- + \frac{\delta_N^2}{4}
\end{equation}
Hence, Problem \eqref{problem:safe_approx_primal} is a linear program over pairs of weight vectors that can be read as measures supported on the grid: the weight $w^-_\tau$ sits at $\tau$ and the weight $w^+_\tau$ at $\tau+\delta_N$, the total mass is one, the first and second moments are constrained as in \eqref{Eq: ambiguity_set} up to the additive term $\frac{\delta_N^2}{4}$ in \eqref{constr:variance_ref}, and the mass in each grid cell is bounded by the envelope mass $\bar{\rho}^+_N(\tau)$ of that cell and hence vanishes as $\delta_N\rightarrow 0$. The correspondence is not exact: in the objective function, both $w^-_\tau$ and $w^+_\tau$ are weighted by the value of the indicator function at the left end point $\tau$ of the cell, whereas in \eqref{problem:safe_approx_primal_exp} and \eqref{constr:variance_ref} the weight $w^+_\tau$ enters at $\tau+\delta_N$.
We further note that exploiting that \eqref{constr:variance} and \eqref{constr:variance_ref} are equivalent, the dual of \eqref{problem:safe_approx_primal} can be written as
\begin{subequations}
	\label{Prob: Dual_Purity_Constraint_ref}
	\begin{align}
		\max_{y \in \mathbb{R}^{5}_{\ge 0}, z \in \mathbb{R}_{\ge 0}^{T_N}}& \langle \left(1,-1, -\mu_{+}, \mu_{-}, - \sigma^2_+ +\mu_+\mu_- - \frac{\delta_N^2}{4} \right),y\rangle - \sum_{\tau \in T_N} \bar{\rho}^+_N(\tau) \boundvariable{\tau} \label{Constr: dual_objective_ref}\\
		\st\ & \label{constr:finite2.2_ref}p_y(\tau) + z_\tau + a\mathbbm{1}_{\tilde{I}(a, x^-, x^+)}(\tau) \geq 0 \text{ for all } \tau \in T_N\\
		\label{constr:finite3.2_ref}&p_y(\tau + \delta_N) + z_\tau + a\mathbbm{1}_{\tilde{I}(a, x^-, x^+)}(\tau) \geq 0 \text{ for all } \tau \in T_N \setminus \{t_{\max}\}.
	\end{align}
\end{subequations}

To conclude the section, we prove a theorem stating that the optimal value of Problem~\eqref{problem:safe_approx_primal} converges to the optimal value of the true inner problem~\eqref{problem:inner} under the assumptions listed below and under an appropriate choice of $\eta_N$.
\begin{theorem}\label{thm:inner_convergence}
	Assume that the ambiguity set $\u$ of \eqref{Eq: ambiguity_set} contains a measure $\P_0$ with $\E_{\P_0}(t) = \mu \coloneqq \frac{\mu_- + \mu_+}{2}$ and $\text{Var}_{\P_0}(t) \leq \sigma^2_+$, where $\sigma^2_+$ denotes the bound on the variance from \eqref{Eq: Sec2_second_moment}, that $\mu_- < \mu_+$ for the bounds on the first moment in \eqref{Eq: Sec2_first_moment}, and that $\rho^{max} > 0$ exists such that the envelope $\bar{\rho}$ of \eqref{Eq: envelope} satisfies $\bar{\rho}(t) \leq \rho^{max}$ for all $t\in T$.
	For $\eta_N \leq \frac{\delta_N^2}{4 \rho^{max} (\bar{T}^3 + 1)}$ with $\bar{T} := t_{\max} - t_{\min}$ and fixed $x^-, x^+ \in T$, the difference of the optimal value of Problem~\eqref{problem:safe_approx_primal} and the optimal value of Problem~\eqref{problem:inner} is not larger than
	\begin{equation}\label{eq:error_bound}
		\kappa_N := |a|\left((5\rho^{max}+1)\delta_N + 2 \max\left\{ \frac{2\delta_N}{(\mu_+ - \mu) + 2\delta_N}, \frac{(2\bar{T} + 1)\delta_N + \frac{5\delta_N^2}{4}}{(\mu^2 - \mu_+\mu_-) + (2\bar{T}+1)\delta_N + \frac{5\delta_N^2}{4}} \right\}\right).
	\end{equation}
\end{theorem}

\begin{proof}
	The optimal value of Problem~\eqref{problem:safe_approx_primal} is by construction not larger than the optimal value of Problem~\eqref{problem:inner}.
	Hence, to prove that the optimal values are close, we start with an optimal solution of \eqref{problem:safe_approx_primal} and construct a feasible solution of \eqref{problem:inner} with a value that is only slightly higher, so that the desired bound is satisfied.
	We note that an optimal solution to \eqref{problem:safe_approx_primal} must exist. Indeed, its feasible set is contained in a simplex by \eqref{constr:prob_measure} and therefore \eqref{problem:safe_approx_primal} bounded. Problem~\eqref{problem:safe_approx_primal} is also feasible: the dual problem $\max \eqref{Constr: dual_objective_ref} \text{ s.t. }\eqref{constr:finite2.2_ref}, \eqref{constr:finite3.2_ref}$ is feasible, as $y=0$ together with $z_\tau = |a|$ satisfies its constraints, and it is bounded from above by the optimal value of \eqref{Prob: Dual_Purity_Constraint}, which by Lemma \ref{Thm: strong_duality} equals the optimal value of \eqref{Prob: Primal_Purity_Constraint_envelope} that is finite due to the existence of $\P_0 \in \u$. Linear programming duality thus yields a feasible, and by boundedness an optimal, solution of \eqref{problem:safe_approx_primal}.

	We first note that the feasibility of \eqref{problem:inner} implies that
	\begin{equation}\label{aux1}
		\int_T \bar{\rho}(t) dt \geq 1.
	\end{equation}
	We define for all $\tau \in T_N$
	\begin{equation*}
		\hat{\rho}^+_N(\tau) := \int_{\tau}^{\tau + \delta_N} \bar{\rho}(t)dt
	\end{equation*}
	and calculate that
	\begin{equation*}
		\bar{\rho}^+_N(\tau) - \hat{\rho}^+_N(\tau) = \int_{\tau - \eta_N}^{\tau + \delta_N + \eta_N} \left(\mathbbm{1}^{\eta_N}_{[\tau,\tau+\delta_N]}(t) - \mathbbm{1}_{[\tau, \tau + \delta_N]}(t)\right) \bar{\rho}(t) dt \leq 2 \eta_N \rho^{\max}.
	\end{equation*}
	We note that \eqref{constr:prob_measure} and \eqref{aux1} imply that
	\begin{equation*}
		\sum_{\tau \in T_N: w^-_\tau + w^+_\tau > \hat{\rho}^+_N(\tau)} (w^-_\tau + w^+_\tau) - \hat{\rho}^+_N(\tau) \leq \sum_{\tau \in T_N: w^-_\tau + w^+_\tau < \hat{\rho}^+_N(\tau)}  \hat{\rho}^+_N(\tau) - (w^-_\tau + w^+_\tau).
	\end{equation*}
	Hence, we can construct $(\tilde{w}^-, \tilde{w}^+)$ by reducing the value of $w^-_\tau + w^+_\tau$ to $\hat{\rho}^+_N(\tau)$ for all $\tau$ where $\hat{\rho}^+_N(\tau)$ is exceeded, and increase the value of $w^-_\tau + w^+_\tau$ for suitable $\tau$ where $\hat{\rho}^+_N(\tau)$ is not reached.
	This procedure shifts probability mass of at most $2\eta_N \rho^{max} |T_N| = 2\eta_N \rho^{max} \frac{\bar{T} + \delta_N}{\delta_N} \leq 4\eta_N \rho^{max} \frac{\bar{T}}{\delta_N}$, as $\delta_N \leq \bar{T}$.
	This implies that the value of the left-hand side of \eqref{problem:safe_approx_primal_exp} changes by at most
	\begin{equation*}
		4\eta_N \rho^{max} \frac{\bar{T}^2}{\delta_N},
	\end{equation*}
	and the value of the left-hand side of \eqref{constr:variance_ref} changes by at most
	\begin{equation*}
		4\eta_N \rho^{max} \frac{\bar{T}^3}{\delta_N}.
	\end{equation*}
	As we have chosen $\eta_N \leq \frac{\delta_N^2}{4 \rho^{max} (\bar{T}^3 + 1)}$, both changes are bounded by $\delta_N$.
	We note that the objective value difference of $w$ and $\tilde{w}$ is bounded by $|a|\delta_N$, i.e.,
	\begin{equation*}
		\left|\sum_{\tau \in T_N} a\mathbbm{1}_{\tilde{I}(a, x^-, x^+)}(\tau) (\tilde{w}_\tau^- + \tilde{w}_\tau^+) - \sum_{\tau \in T_N} a\mathbbm{1}_{\tilde{I}(a, x^-, x^+)}(\tau) (w_\tau^- + w_\tau^+)\right| \leq |a|\delta_N.
	\end{equation*}

	The next step is to construct a measure $\P_1$ that has the property that its density is a multiple of the envelope $\bar{\rho}(t)$ on the intervals $[\tau, \tau + \delta_N)$ for all $\tau \in T_N$, and that has a mass on these intervals that is identical to $\tilde{w}^-_\tau + \tilde{w}^+_\tau$. We obtain this by setting
	\begin{equation*}
		\rho(t) := \frac{\tilde{w}^-_\tau + \tilde{w}^+_\tau}{\hat{\rho}^+_N(\tau)} \bar{\rho}(t) \text{ for } t\in [\tau, \tau + \delta_N) \text{ for all }\tau \in T_N,
	\end{equation*}
	where we set $\rho := 0$ on the cells with $\hat{\rho}^+_N(\tau) = 0$, on which $\tilde{w}^-_\tau + \tilde{w}^+_\tau = 0$ by construction. This includes the cell $[t_{\max}, t_{\max}+\delta_N)$, on which $\bar{\rho}$ vanishes, so that $\P_1$ is concentrated on $T$.
	We note that this procedure shifts probability mass of up to $1$ by at most $\delta_N$, starting from the measure that is induced by the vector $\tilde{w}$ that has its mass exclusively on the points in $T_N$.
	To this end, consider
	\begin{equation*}
		\int_{[\tau, \tau + \delta_N)} t d\P_1 - \left(\tau \tilde{w}^- + (\tau + \delta_N) \tilde{w}^+\right) \leq (\tau + \delta_N) (\tilde{w}^- + \tilde{w}^+) - \tau (\tilde{w}^- + \tilde{w}^+) = \delta_N(\tilde{w}^- + \tilde{w}^+),
	\end{equation*}
	and
	\begin{equation*}
		\int_{[\tau, \tau + \delta_N)} t d\P_1 - \left(\tau \tilde{w}^- + (\tau + \delta_N) \tilde{w}^+\right) \geq \tau (\tilde{w}^- + \tilde{w}^+) - (\tau + \delta_N) (\tilde{w}^- + \tilde{w}^+) = - \delta_N(\tilde{w}^- + \tilde{w}^+).
	\end{equation*}
	for all $\tau \in T_N$. Exploiting these inequalities, we obtain
	\begin{equation*}
		\left| \int_T t d\P_1 - \sum_{\tau \in T_N} \tau \tilde{w}^- + (\tau + \delta_N) \tilde{w}^+ \right| \leq \sum_{\tau \in T_N} \delta_N (\tilde{w}^- + \tilde{w}^+) = \delta_N,
	\end{equation*}
	as $\sum_{\tau \in T_N} \tilde{w}^- + \tilde{w}^+ = 1$.
	This implies that the value of the left-hand side of \eqref{problem:safe_approx_primal_exp} changes by at most $\delta_N$.
	With a similar argument, we can bound the change of the value of the left-hand side of \eqref{constr:variance_ref} and thus \eqref{constr:variance}. More general, it holds for measurable $f$ that
	\begin{align*}
		&\left|\int_{[\tau, \tau + \delta_N)} f(t) d\P_1 - \left(f(\tau) \tilde{w}^- + f(\tau + \delta_N) \tilde{w}^+\right)\right| \\
		&\qquad \leq \left(\sup_{t\in [\tau, \tau+\delta_N]} f(t) - \inf_{t \in [\tau, \tau+\delta_N]} f(t)\right) \left(\tilde{w}^- + \tilde{w}^+\right).
	\end{align*}
	We calculate that for $\mu=\frac{\mu_-+\mu_+}{2}$
	\begin{align*}
		& \sup_{t \in [\tau, \tau+\delta_N]} t^2 - 2\mu t - \inf_{t \in [\tau, \tau + \delta_N]} t^2 - 2 \mu t \\
		& \qquad = \sup_{t \in [\tau, \tau + \delta_N]} (t - \mu)^2 -  \inf_{t \in [\tau, \tau + \delta_N]} (t - \mu)^2 \\
		& \qquad \leq 2\delta_N |\tau - \mu| + \delta_N^2 \\
		& \qquad \leq 2\delta_N \bar{T} + \delta_N^2,
	\end{align*}
	where the first inequality holds since the difference equals $\left|(\tau - \mu + \delta_N)^2 - (\tau - \mu)^2\right|$ if $\mu\notin(\tau,\tau+\delta_N)$ and is at most $\delta_N^2$ otherwise, and the second one holds as $\mu = \E_{\P_0}(t) \in T$, $T$ being an interval. Hence,
	\begin{align*}
		&\left| \int_T t^2 - 2 \mu t d\P_1 - \sum_{\tau \in T_N} (\tau^2 - 2\mu \tau) \tilde{w}^- + \left((\tau + \delta_N)^2 - 2\mu(\tau + \delta_N)\right) \tilde{w}^+ \right| \\
		& \qquad \leq \sum_{\tau \in T_N} (2\delta_N\bar{T} + \delta_N^2) (\tilde{w}^- + \tilde{w}^+) = 2\delta_N\bar{T} + \delta_N^2,
	\end{align*}
	as $\sum_{\tau \in T_N} \tilde{w}^- + \tilde{w}^+ = 1$.
	We note that, as the function $\mathbbm{1}_{\tilde{I}(a, x^-, x^+)}$ is constant in all but two intervals $[\tau, \tau + \delta_N)$, the following difference is bounded as well:
	\begin{align*}
		& \left| \int_T a\mathbbm{1}_{\tilde{I}(a, x^-, x^+)} d\P_1 - \sum_{\tau \in T_N} a\mathbbm{1}_{\tilde{I}(a, x^-, x^+)}(\tau) (\tilde{w}_\tau^- + \tilde{w}_\tau^+)\right|\\
		& \qquad =  \left| \sum_{\tau \in T_N} \int_{[\tau, \tau + \delta_N]} a\mathbbm{1}_{\tilde{I}(a, x^-, x^+)} d\P_1 - a\mathbbm{1}_{\tilde{I}(a, x^-, x^+)}(\tau) (\tilde{w}_\tau^- + \tilde{w}_\tau^+)\right|\\
		& \qquad \leq |a| \sum_{\tau \in T_N} \left|  \int_{[\tau, \tau + \delta_N]} \mathbbm{1}_{\tilde{I}(a, x^-, x^+)} d\P_1 - \mathbbm{1}_{\tilde{I}(a, x^-, x^+)}(\tau) (\tilde{w}_\tau^- + \tilde{w}_\tau^+)\right|\\
		& \qquad \leq |a|2\rho^{max}\delta_N,
	\end{align*}
	This holds because only two summands of the last sum are nonzero and each summand is bounded from above by $\rho^{max}\delta_N$, as
	\begin{align*}
		&0 \leq \int_{[\tau, \tau + \delta_N]} \mathbbm{1}_{\tilde{I}(a, x^-, x^+)} d\P_1 = \int_{[\tau, \tau + \delta_N]} \mathbbm{1}_{\tilde{I}(a, x^-, x^+)}(t) \frac{\tilde{w}^-_\tau + \tilde{w}^+_\tau}{\hat{\rho}^+_N(\tau)} \bar{\rho}(t) dt\\
		&\leq \int_{[\tau, \tau + \delta_N]} \bar{\rho}(t) dt \leq \rho^{max} \delta_N,
	\end{align*}
	and
	\begin{equation*}
		0 \leq \mathbbm{1}_{\tilde{I}(a, x^-, x^+)}(\tau) (\tilde{w}_\tau^- + \tilde{w}_\tau^+) \leq (\tilde{w}_\tau^- + \tilde{w}_\tau^+) \leq \hat{\rho}^+_N(\tau) = \int_{[\tau, \tau + \delta_N]} \bar{\rho}(t) dt \leq \rho^{max} \delta_N.
	\end{equation*}
	We note that $\langle t, \P_1 \rangle \in [\mu_- - 2 \delta_N, \mu_+ + 2 \delta_N]$, and $\langle t^2 - (\mu_-+\mu_+) t, \P_1 \rangle \leq \sigma^2_+ - \mu_+ \mu_- + \frac{\delta_N^2}{4} + \delta_N + 2\bar{T} \delta_N + \delta_N^2$, as the measure induced by $w$ violates Constraint~\eqref{Eq: Sec2_second_moment_true} by at most $\frac{\delta_N^2}{4}$, going over to the measure induced by $\tilde{w}$ increases this violation by at most $\delta_N$, and going over to $\P_1$ increases this violation by at most $2\delta_N\bar{T} + \delta_N^2$.
	We further note that $\langle t, \P_0 \rangle = \mu$, and $\langle t^2 - 2\mu t, \P_0 \rangle \leq \sigma^2_+ - \mu^2$.
	Hence, if we choose
	\begin{equation*}
		\lambda := \max\left\{ \frac{2\delta_N}{(\mu_+ - \mu) + 2\delta_N}, \frac{(2\bar{T} + 1)\delta_N + \frac{5\delta_N^2}{4}}{(\mu^2 - \mu_+\mu_-) + (2\bar{T}+1)\delta_N + \frac{5\delta_N^2}{4}} \right\},
	\end{equation*}
	we can calculate that
	\begin{align*}
		\langle t, \lambda \P_0 + (1 - \lambda) \P_1 \rangle \leq \lambda \mu + (1 - \lambda) (\mu_+ + 2 \delta_N)\\
		\overset{\text{(*)}}{\leq} \frac{\mu 2\delta_N}{(\mu_+ - \mu) + 2\delta_N} + \frac{(\mu_+ - \mu) (\mu_+ + 2\delta_N)}{(\mu_+ - \mu) + 2\delta_N}\\
		= \frac{\mu2\delta_N + \mu_+^2 - \mu_+\mu + \mu_+ 2\delta_N - \mu 2\delta_N}{(\mu_+ - \mu) + 2\delta_N}\\
		= \frac{\mu_+ (\mu_+ - \mu + 2\delta_N) }{(\mu_+ - \mu) + 2\delta_N} = \mu_+.
	\end{align*}
	We note that (*) holds because $\mu < \mu_+ + 2 \delta_N$ which in turn implies that $\lambda \left(\mu - (\mu_+ + 2 \delta_N)\right) + (\mu_+ + 2 \delta_N) \leq \frac{2\delta_N\left(\mu - (\mu_+ + 2 \delta_N)\right)}{(\mu_+ - \mu) + 2 \delta_N} + (\mu_+ + 2 \delta_N)$.
	With an analogue argument we can show that $\langle t, \lambda \P_0 + (1 - \lambda) \P_1 \rangle \geq \mu_-$.
	We can further calculate that with $A := (2\bar{T}+1)\delta_N + \frac{5\delta_N^2}{4}$ and $B := \mu^2 - \mu_+\mu_-$
	\begin{align*}
		\langle t^2 - (\mu_-+\mu_+) t, \lambda \P_0 + (1 - \lambda) \P_1 \rangle\\
		\leq \frac{(\sigma^2_+ - \mu^2)A}{A + B} + \frac{(\sigma^2_+ - \mu_+\mu_- + A)B }{A + B}\\
		= \frac{\sigma^2_+A - \mu^2A + \sigma^2_+B - \mu_+\mu_- B + AB}{A+B}\\
		= \frac{\sigma^2_+A - (B + \mu_+\mu_-)A + \sigma^2_+B - \mu_+\mu_- B + AB}{A+B}\\
		= \frac{\sigma^2_+A - BA - \mu_+\mu_- A + \sigma^2_+B - \mu_+\mu_- B + AB}{A+B}\\
		= \frac{(\sigma^2_+ - \mu_+\mu_-)(A+B)}{A+B}\\
		= \sigma^2_+ - \mu_+\mu_-,
	\end{align*}
	where the first inequality is again, by observing that $\sigma^2_+-\mu^2 \leq \sigma^2_+-\mu_-\mu_+ +A$ implies that the estimate $\lambda \geq \frac{A}{A+B}$ provides an upper bound. We conclude that $\P := \lambda \P_0 + (1 - \lambda) \P_1$ is feasible for the true inner Problem~\eqref{problem:inner}.

	We note that
	\begin{equation*}
		\left|\int_T a\mathbbm{1}_{[x^-, x^+]} d\P - \int_T a\mathbbm{1}_{\tilde{I}(a, x^-, x^+)}d\P\right| \leq |a| 3 \delta_N \rho^{max}.
	\end{equation*}

	We can further calculate that the value difference of $\P$ in \eqref{problem:inner} and $w$ in \eqref{problem:safe_approx_primal} can be bounded from above by
	\begin{align*}
		& \left| \int_T a\mathbbm{1}_{[x^-, x^+]} d\P - \sum_{\tau \in T_N} a\mathbbm{1}_{\tilde{I}(a, x^-, x^+)}(\tau) (w_\tau^- + w_\tau^+) \right|\\
		& \qquad \leq
		\left| \int_T a\mathbbm{1}_{[x^-, x^+]} d\P - \int_T a\mathbbm{1}_{\tilde{I}(a, x^-, x^+)} d\P \right|\\
		& \qquad + \lambda \left|\int_T a\mathbbm{1}_{\tilde{I}(a, x^-, x^+)} d\P_0\right|\\
		& \qquad + \left| \int_T a\mathbbm{1}_{\tilde{I}(a, x^-, x^+)} d\P_1 - \sum_{\tau \in T_N} a\mathbbm{1}_{\tilde{I}(a, x^-, x^+)}(\tau) (\tilde{w}_\tau^- + \tilde{w}_\tau^+)\right|\\
		& \qquad + \left|\sum_{\tau \in T_N} a\mathbbm{1}_{\tilde{I}(a, x^-, x^+)}(\tau) (\tilde{w}_\tau^- + \tilde{w}_\tau^+) - \sum_{\tau \in T_N} a\mathbbm{1}_{\tilde{I}(a, x^-, x^+)}(\tau) (w_\tau^- + w_\tau^+)\right|\\
		& \qquad + \lambda \left|\sum_{\tau \in T_N} a\mathbbm{1}_{\tilde{I}(a, x^-, x^+)}(\tau) (w_\tau^- + w_\tau^+)\right|\\
		& \qquad \leq |a|3\delta_N\rho^{max} + \lambda|a| + |a|2\rho^{max}\delta_N + |a|\delta_N + \lambda|a|\\
		& \qquad \leq |a|\left((5\rho^{max}+1)\delta_N + 2 \max\left\{ \frac{2\delta_N}{(\mu_+ - \mu) + 2\delta_N}, \frac{(2\bar{T} + 1)\delta_N + \frac{5\delta_N^2}{4}}{(\mu^2 - \mu_+\mu_-) + (2\bar{T}+1)\delta_N + \frac{5\delta_N^2}{4}} \right\}\right).
	\end{align*}
	As $w$ is optimal for Problem~\eqref{problem:safe_approx_primal}, and $\P$ feasible for Problem~\eqref{problem:inner}, this proves the statement.
\end{proof}
Hence, we note that Theorem \ref{thm:inner_convergence} shows that the inner problem~\eqref{problem:inner} can be approximated to any desired accuracy that depends on the value of $\delta_N$.
Two properties of the bound \eqref{eq:error_bound} will be used in the next subsection. First, $\kappa_N$ does not depend on $x^-$ and $x^+$, i.e., the approximation error is uniform over the upper-level decisions. Second, $\kappa_N \rightarrow 0$ as $\delta_N \rightarrow 0$, because $\mu_+ - \mu = \frac{\mu_+-\mu_-}{2} > 0$ and $\mu^2 - \mu_+\mu_- = \left(\frac{\mu_+-\mu_-}{2}\right)^2 > 0$ by the assumption $\mu_-<\mu_+$.

\subsection{Consequences for the outer problem}\label{sec:outer}

Theorem \ref{thm:inner_convergence} is a statement about the adversarial problem for fixed $x^-,x^+$. We now transfer it to Problem \eqref{problem:basic2}. For $s\in S$ let
\begin{equation*}
	\text{val}^s(x^-,x^+) := \min_{\mathbbm{P}^s \in \u^s} a^s \mathbbm{P}^s\left([x^-, x^+]\right)
\end{equation*}
denote the optimal value of the adversarial problem $\eqref{problem:inner}^s$, let $\text{val}^s_N(x^-,x^+)$ denote the optimal value of its approximation $\eqref{problem:safe_approx_primal}^s$, and let $\kappa^s_N$ be the corresponding bound \eqref{eq:error_bound}. We consider the approximation
\begin{subequations}\label{problem:safe_approx_outer}
	\begin{align}
		\max_{(x^-,x^+, \tilde{x}) \in C} \quad & c(x^-,x^+, \tilde{x})\\
		\label{constr:safe_approx_outer}\st\ & \sum_{s \in S} \text{val}^s_N(x^-,x^+) \geq b
	\end{align}
\end{subequations}
of Problem \eqref{problem:basic2} and denote by $\text{val}^{*}$ and $\text{val}_N$ the optimal values of \eqref{problem:basic2} and \eqref{problem:safe_approx_outer}, with the convention $\sup \emptyset = -\infty$.
In particular, \eqref{problem:safe_approx_outer} is a safe approximation of \eqref{problem:basic2}, see \Cref{cor:safety} below, and admits a MIP formulation, see \Cref{Thm: MIP_onedim_safe} in \Cref{sec:mip}.

\begin{coroll}\label{cor:safety}
	Every feasible point of \eqref{problem:safe_approx_outer} is feasible for \eqref{problem:basic2}, and consequently $\text{val}_N \leq \text{val}^{*}$.
\end{coroll}

\begin{proof}
	By construction, $\text{val}^s_N(x^-,x^+) \leq \text{val}^s(x^-,x^+)$ for all $s\in S$ and all $x^-,x^+\in T$, see the first line of the proof of Theorem \ref{thm:inner_convergence}. Hence \eqref{constr:safe_approx_outer} implies $b \leq \sum_{s\in S}\text{val}^s_N(x^-,x^+) \leq \sum_{s\in S}\text{val}^s(x^-,x^+)$, which is \eqref{problem:basic_prob_constr2}. As both problems have the same objective function and the same set $C$, the claim on the optimal values follows.
\end{proof}

The following theorem shows that under a strict feasibility assumption $\text{val}_N$ converges to $\text{val}^{*}$.

\begin{theorem}\label{thm:outer_convergence}
	Assume that the hypotheses of Theorem \ref{thm:inner_convergence} are satisfied for every $s\in S$, that Problem \eqref{problem:basic2} is feasible with $\text{val}^{*}<\infty$, and that the following strict feasibility condition holds:
	\begin{equation}\label{eq:strict_feasibility}
		\begin{aligned}
			&\text{for every } \varepsilon>0 \text{ there is } x_\varepsilon = (x^-_\varepsilon,x^+_\varepsilon,\tilde{x}_\varepsilon) \in C \text{ with}\\
			&\qquad c(x_\varepsilon)\geq \text{val}^{*}-\varepsilon \quad \text{and} \quad \sum_{s\in S}\text{val}^s(x^-_\varepsilon,x^+_\varepsilon) > b.
		\end{aligned}
	\end{equation}
	Then $\text{val}_N \rightarrow \text{val}^{*}$ as $\delta_N \rightarrow 0$.
\end{theorem}

\begin{proof}
	By Corollary \ref{cor:safety}, $\text{val}_N \leq \text{val}^{*}$ for every $N$, so it suffices to show that $\liminf_{\delta_N\rightarrow 0}\text{val}_N \geq \text{val}^{*}-\varepsilon$ for every $\varepsilon>0$.
	Let $\varepsilon>0$ and let $x_\varepsilon$ be as in \eqref{eq:strict_feasibility}. Then
	\begin{equation*}
		\gamma_\varepsilon := \sum_{s\in S}\text{val}^s(x^-_\varepsilon,x^+_\varepsilon) - b > 0.
	\end{equation*}
	By Theorem \ref{thm:inner_convergence}, applied for every $s\in S$ with the fixed pair $x^-_\varepsilon,x^+_\varepsilon$,
	\begin{equation*}
		\sum_{s\in S}\text{val}^s_N(x^-_\varepsilon,x^+_\varepsilon) \geq \sum_{s\in S}\text{val}^s(x^-_\varepsilon,x^+_\varepsilon) - \sum_{s\in S}\kappa^s_N = b + \gamma_\varepsilon - \sum_{s\in S}\kappa^s_N.
	\end{equation*}
	Since $S$ is finite and $\kappa^s_N \rightarrow 0$ for every $s\in S$, there is an $N_\varepsilon$ such that $\sum_{s\in S}\kappa^s_N \leq \gamma_\varepsilon$ for all $N\geq N_\varepsilon$. For those $N$, the point $x_\varepsilon$ satisfies \eqref{constr:safe_approx_outer} and lies in $C$. Hence it is feasible for \eqref{problem:safe_approx_outer} and
	\begin{equation*}
		\text{val}_N \geq c(x_\varepsilon) \geq \text{val}^{*}-\varepsilon.
	\end{equation*}
	As $\varepsilon>0$ was arbitrary, the claim follows.
\end{proof}

For finite $N$, an upper bound on $\text{val}^{*}$ can be computed as well, so that the distance between $\text{val}_N$ and $\text{val}^{*}$ can be estimated.

\begin{coroll}\label{cor:enclosure}
	Assume that the hypotheses of Theorem \ref{thm:inner_convergence} are satisfied for every $s\in S$, and let $\text{val}^{\uparrow}_N$ denote the optimal value of the problem obtained from \eqref{problem:safe_approx_outer} by replacing the right-hand side $b$ of \eqref{constr:safe_approx_outer} by $b-\sum_{s\in S}\kappa^s_N$. Then
	\begin{equation*}
		\text{val}_N \leq \text{val}^{*} \leq \text{val}^{\uparrow}_N.
	\end{equation*}
\end{coroll}

\begin{proof}
	The first inequality is Corollary \ref{cor:safety}. For the second one, let $x=(x^-,x^+,\tilde{x})$ be feasible for \eqref{problem:basic2}. By Theorem \ref{thm:inner_convergence},
	\begin{equation*}
		\sum_{s\in S}\text{val}^s_N(x^-,x^+) \geq \sum_{s\in S}\text{val}^s(x^-,x^+) - \sum_{s\in S}\kappa^s_N \geq b - \sum_{s\in S}\kappa^s_N,
	\end{equation*}
	so that $x$ is feasible for the modified problem. Hence the feasible set of \eqref{problem:basic2} is contained in that of the modified problem, and the optimal values are ordered accordingly.
\end{proof}

We conclude the section with an example that shows that the strict feasibility condition \eqref{eq:strict_feasibility} cannot be dropped from Theorem \ref{thm:outer_convergence}.

\begin{example}\label{example:no_convergence_to_expect_chromatography}
	We assume $T^1 := T^2 := T := [0, 5]$ and $\bar{\rho}^1 = 2\mathbbm{1}_{[1, 2]}$, $\bar{\rho}^2 = \frac{1}{4}\mathbbm{1}_{[0, 5]}$.
	We consider the problem
	\begin{align}
		\max_{x^- < x^+ \in T} \quad & x^+ - x^-\\
		\text{s.t.}\quad & \min_{\P^1 \in \u^1} a^1 \P^1\left([x^-, x^+]\right) +   \min_{\P^2 \in \u^2} a^2 \P^2\left([x^-, x^+]\right) \geq 0
	\end{align}
	with $a^1 = 0.2$, $a^2 = -0.8$, and $$\u^s = \left\{ \P \in \Mcal(T)_{\geq 0} \mid \P(T) = 1, \rho_\P(t) \leq \bar{\rho}^s(t) \text{ for all }t\in T \right\}$$ for $s= 1, 2$, with $\rho_\P$ denoting the probability density of measure $\P$ with respect to the Lebesgue measure.
	We note that
	\begin{align*}
		\min_{\P^2 \in \u^2} a^2 \P^2\left([x^-, x^+]\right) & =-0.8 \max_{\P^2 \in \u^2} \int_{x^-}^{x^+} \rho_{\P^2}(t) dt\\
		& = -0.8 \min \left\{1, \frac{1}{4}(x^+ - x^-)\right\},
	\end{align*}
	where $\rho_{\P^2}$ is chosen as $\frac{1}{4}\mathbbm{1}_{[a, b]}$ with $b-a=4$ in a way to maximize the intersection of $[a,b]\cap [x^-,x^+]$. Since every density in $\u^2$ is bounded by $\frac{1}{4}$, we have $\P^2([x^-,x^+])\leq \min\{1,\frac{1}{4}(x^+-x^-)\}$ for all $\P^2\in\u^2$, and this choice attains the bound. We note that for every feasible $x^-,x^+$, we have that
	\begin{align}
			\min_{\P^1 \in \u^1} & a^1 \P^1\left([x^-, x^+]\right) \geq - \min_{\P^2 \in \u^2} a^2 \P^2\left([x^-, x^+]\right)\notag\\
			& \Leftrightarrow \min_{\P^1 \in \u^1} 0.2 \P^1\left([x^-, x^+]\right) \geq 0.8 \min \left\{1, \frac{1}{4}(x^+ - x^-)\right\}\notag\\
			& \Leftrightarrow \min_{\P^1 \in \u^1}  \int_{x^-}^{x^+} \rho_{\P^1}(t) dt \geq \min\left\{4,(x^+ - x^-)\right\}\label{eq:lb_P1}
	\end{align}
	Since the left-hand side of \eqref{eq:lb_P1} is a probability and hence at most $1$, we obtain $x^+-x^-\leq 1$.
	If the length of the interval $[x^-, x^+] \cap [1, 2]$ is contained in $\left(\frac{1}{2},1\right)$, the measure with density
	\begin{equation*}
		\rho := 2\mathbbm{1}_{[1, 2] \setminus [x^-, x^+]} + \lambda \mathbbm{1}_{[1, 2] \cap [x^-, x^+]}
	\end{equation*}
	with $\lambda < 1$ such that $\rho$ is a probability density is feasible for $\min_{\P^1 \in \u^1} a^1 \P^1\left([x^-, x^+]\right)$ and assigns at most $\lambda(x^+ - x^-)$ to $[x^-,x^+]$.
	If the length of the interval $[x^-, x^+] \cap [1, 2]$ is contained in $\left[0,\frac{1}{2}\right]$, the measure with density
	\begin{equation*}
		\rho := \lambda \mathbbm{1}_{[1, 2] \setminus [x^-, x^+]}
	\end{equation*}
	with $\lambda \geq 1$ such that $\rho$ is a probability density is feasible for $\min_{\P^1 \in \u^1} a^1 \P^1\left([x^-, x^+]\right)$, with a value of $0$.

	Hence, for every point $x^- < x^+$ except for $x^-=1, x^+=2$, we have that $\min_{\P^1 \in \u^1} a^1 \P^1\left([x^-, x^+]\right) \leq \lambda(x^+-x^-)$ with $0\leq \lambda < 1$, which violates \eqref{eq:lb_P1} and thus renders $x^-,x^+$ infeasible. Consequently, the problem has exactly one feasible point, namely $(x^-, x^+) = (1, 2)$, which fulfills the DRO constraint with equality.
	If we instead use strict minorants of the inner problem, the outer problem becomes infeasible.
\end{example}

In this example the unique feasible point satisfies the distributionally robust constraint with equality, so that no strictly feasible point exists and \eqref{eq:strict_feasibility} fails for every $\varepsilon>0$. Consequently, an approximation that underestimates one of the two adversarial values strictly renders \eqref{problem:safe_approx_outer} infeasible, i.e., $\text{val}_N=-\infty$, while $\text{val}^{*}=1$. No convergence statement of the type of Theorem \ref{thm:outer_convergence} can therefore hold without an assumption of this kind.

Problem \eqref{problem:safe_approx_outer} is not yet in a form that can be passed to a standard solver: its constraint contains the optimal values $\text{val}^s_N(x^-,x^+)$ of the linear programs $\eqref{problem:safe_approx_primal}^s$, which depend on the decisions $x^-,x^+$ through indicator functions of intervals. In the next section, we represent each of these linear programs by its dual \eqref{Prob: Dual_Purity_Constraint_ref}, so that \eqref{constr:safe_approx_outer} holds if and only if dual feasible variables with a sufficiently large objective value exist, and add these variables to the decisions of the outer problem. Encoding the remaining indicator functions by binary variables then yields an exact mixed-integer formulation of \eqref{problem:safe_approx_outer}.
The problem defining $\text{val}^{\uparrow}_N$ differs only in the right-hand side of \eqref{constr:safe_approx_outer}.
Hence, any reformulation of \eqref{problem:safe_approx_outer} carries over to the problem defining $\text{val}^{\uparrow}_N$.

\section{A Mixed-Integer Formulation of the Safe Approximation}\label{sec:mip}
We have gathered all the ingredients to provide a MIP formulation for the safe approximation of Problem~\eqref{problem:basic2}.
To this end, recall that, based on the results in Section \ref{sec:safe_approx}, the following is a safe approximation of Problem \eqref{problem:basic2}:
\begin{subequations}\label{problem:safe_approx_after_sec3}
	\begin{align}
		\max \quad & c(x^-,x^+, \tilde{x})\\
		\st\ & \langle \left(1,-1, -\mu_{+}, \mu_{-}, - \sigma^2_+ +\mu_+\mu_- - \frac{\delta_N^2}{4} \right),y\rangle - \sum_{\tau \in T_N} \bar{\rho}^+_N(\tau) \boundvariable{\tau} \geq b,\label{constr:safe_approx_after_sec3_constr1}\\
		& p_y(\tau) + z_\tau + a\mathbbm{1}_{\tilde{I}(a, x^-, x^+)}(\tau) \geq 0 \text{ for all } \tau \in T_N,\label{constr:safe_approx_after_sec3_constr2}\\
		& p_y(\tau + \delta_N) + z_\tau + a\mathbbm{1}_{\tilde{I}(a, x^-, x^+)}(\tau) \geq 0 \text{ for all } \tau \in T_N \setminus \{t_{\max}\},\label{constr:safe_approx_after_sec3_constr3}\\
		& (x^-,x^+, \tilde{x}) \in C, y \in \mathbb{R}^{5}_{\ge 0}, z \in \mathbb{R}_{\ge 0}^{T_N}.
	\end{align}
\end{subequations}
Hence, each feasible solution to \eqref{problem:safe_approx_after_sec3} is feasible to \eqref{problem:basic2}.
Furthermore, we can calculate an upper bound on its optimality gap by adapting the right-hand side of \eqref{constr:safe_approx_after_sec3_constr1} according to the bound given by Theorem \ref{thm:inner_convergence}.
We note that \eqref{problem:safe_approx_after_sec3} is MIP-representable if we can replace $\mathbbm{1}_{\tilde{I}(a, x^-, x^+)}(\tau)$ by binary variables $\chi_\tau$ that encode $\mathbbm{1}_{\tilde{I}(a, x^-, x^+)}(\tau)$ appropriately for all feasible $x^-, x^+ \in T$.

We note that the system
\begin{subequations}\label{28}
	\begin{align}
		\label{28a}& \sum_{\tau\in T_N} \lift{\tau}{-} = 1\\
		\label{28b}& \sum_{\tau\in T_N} \lift{\tau}{+} = 1\\
		\label{28c}& \sum_{\tau\in T_N} \tau \lift{\tau}{-} \leq x^- \leq \sum_{\tau\in T_N} (\tau + \delta_N) \lift{\tau}{-}\\
		\label{28d}& \sum_{\tau\in T_N} \tau \lift{\tau}{+} \leq x^+ \leq \sum_{\tau\in T_N} (\tau + \delta_N) \lift{\tau}{+}\\
		\label{28e}&\Delta^-,\Delta^+ \in \{0,1\}^{T_N}
	\end{align}
\end{subequations}
ascertains that $\Delta^-_\tau = 1$ and $\Delta^+_\tau = 1$ are equivalent to $x^- \in [\tau, \tau + \delta_N]$ and $x^+ \in [\tau, \tau + \delta_N]$ respectively.
Next, we set up two systems that encode $\mathbbm{1}_{\tilde{I}(a, x^-, x^+)}(\tau)$ with binary variables such that $a\chi_{\tau} \leq a\mathbbm{1}_{\tilde{I}(a, x^-, x^+)}(\tau)$ in order to guarantee that the resulting MIP satisfies \eqref{constr:safe_approx_after_sec3_constr2} and \eqref{constr:safe_approx_after_sec3_constr3} and consequently remains a safe approximation.
Hence, we have to distinguish between $a > 0$ and $a < 0$.
Assuming that $x, \Delta$ fulfills System~\eqref{28}, and $a > 0$, we aim to show that the system
\begin{subequations}\label{29}
	\begin{align}
		\label{29a}& \chi_{\tau} \leq \sum_{t \in T_N, t \leq \tau - 2\delta_N} \Delta^-_t && \forall \tau\in T_N\\
		\label{29b}& \chi_{\tau} \leq \sum_{t \in T_N, t \geq \tau + 2\delta_N} \Delta^+_t && \forall \tau\in T_N\\
		& \chi\in \{0,1\}^{T_N}&&
	\end{align}
\end{subequations}
ascertains that $\chi_\tau = 0$ whenever $\mathbbm{1}_{\tilde{I}(a, x^-, x^+)}(\tau) = 0$ or equivalently $\tau \notin [x^-+\delta_N,x^+-2\delta_N]$.
We first analyze \eqref{29a}: It holds that
\begin{equation*}
	\tau < x^- + \delta_N \Leftrightarrow \tau - \delta_N < x^- \Rightarrow \sum_{t \in T_N, t \leq \tau - 2\delta_N} \Delta^-_t = 0
\end{equation*}
and
\begin{equation*}
	\tau > x^- + \delta_N \Leftrightarrow \tau - \delta_N > x^- \Rightarrow \sum_{t \in T_N, t \leq \tau - 2\delta_N} \Delta^-_t = 1
\end{equation*}
If $x^-=\tau-\delta_N$, the term $\sum_{t \in T_N, t \leq \tau - 2\delta_N} \Delta^-_t$ can attain either $0$ or $1$, which implies that $\chi_\tau$ can attain the value $1$ if and only if $\tau \geq x^-+\delta_N$. Similarly, if we analyze \eqref{29b}, it holds that
$\chi_\tau$ can attain the value $1$ if and only if $\tau \leq x^+ - 2\delta_N$. Hence, $\chi_\tau$ is forced to $0$ if
$\tau \notin \tilde{I}(a, x^-, x^+)$ for $a > 0$.

Similarly, if $a < 0$, we show that the system
\begin{subequations}\label{30}
	\begin{align}
		\label{30a}& \chi_{\tau} \geq \sum_{t \in T_N, t \leq \tau +\delta_N} \Delta^-_t + \sum_{t \in T_N, t \geq \tau - \delta_N} \Delta^+_t - 1 && \forall \tau\in T_N\\
		& \chi \in \{0,1\}^{T_N}&&
	\end{align}
\end{subequations}
ascertains that $\chi_\tau$ is forced to $1$ if and only if $\mathbbm{1}_{\tilde{I}(a, x^-, x^+)}(\tau) = 1$,
i.e., if $\tau \in \tilde{I}(a, x^-, x^+) = (x^- -2\delta_N,x^+ +\delta_N)$.
It holds that
\begin{equation*}
	\tau > x^- - 2\delta_N \Leftrightarrow \tau + 2\delta_N > x^- \Rightarrow \sum_{t \in T_N, t \leq \tau + \delta_N} \Delta^-_t = 1,
\end{equation*}
and
\begin{equation*}
	\tau < x^- - 2\delta_N \Leftrightarrow \tau + 2\delta_N < x^- \Rightarrow \sum_{t \in T_N, t \leq \tau + \delta_N} \Delta^-_t = 0,
\end{equation*}
as well as both, $0$ and $1$ being attainable for $\sum_{t \in T_N, t \leq \tau + \delta_N} \Delta^-_t$ if $\tau = x^- - 2\delta_N$.
Again, with similar arguments, one can conclude that $\sum_{t \in T_N, t \geq \tau - \delta_N} \Delta^+_t$ can attain the value $0$ if and only if $\tau \geq x^+ + \delta_N$.
Hence, $\chi_\tau$ is forced to $1$ if and only if $\tau \in \tilde{I}(a, x^-, x^+)$ for $a < 0$.

\begin{remark}\label{remark:alt_systems}
	We would like to point out that in order to enhance numerical efficiency, it is possible to significantly decrease the number of non-zeros in Systems \eqref{29} and \eqref{30} by the equivalent systems
	\begin{subequations}\label{29_alt}
		\begin{align}
			\label{29a_alt}&\chi_{t_{\min}} = \chi_{t_{\min}+\delta_N} = 0, \chi_\tau \leq \chi_{\tau - \delta_N} + \Delta^-_{\tau - 2 \delta_N} \text{ for all } \tau \geq t_{\min} + 2 \delta_N,\\
			\label{29b_alt}&\chi_{t_{\max}} = \chi_{t_{\max} - \delta_N} = 0, \chi_\tau \leq \chi_{\tau + \delta_N} + \Delta^+_{\tau + 2 \delta_N} \text{ for all } \tau \leq t_{\max} - 2 \delta_N.
		\end{align}
	\end{subequations}
	and
	\begin{subequations}\label{30_alt}
		\begin{align}
			\label{30a_alt}\chi^-_{t_{\min}} \geq \Delta^-_{t_{\min}} + \Delta^-_{t_{\min}+\delta_N},~\chi^-_{\tau} \geq \chi^-_{\tau - \delta_N} + \Delta^-_{\tau + \delta_N} \text{ for all }\tau \geq t_{\min}+\delta_N\\
			\chi^+_{t_{\max}} \geq \Delta^+_{t_{\max}} + \Delta^+_{t_{\max} - \delta_N},~\chi^+_{\tau} \geq \chi^+_{\tau + \delta_N} + \Delta^+_{\tau - \delta_N} \text{ for all }\tau \leq t_{\max} - \delta_N\\
			\chi_\tau \geq \chi^-_\tau + \chi^+_\tau - 1\text{ for all }\tau \in T_N\\
			\chi^-, \chi^+ \in \{0, 1\}^{T_N}
		\end{align}
	\end{subequations}
	respectively. However, in the latter case this comes at the expense of introducing additional binary variables.
\end{remark}
With the help of Systems~\eqref{28}, \eqref{29_alt} and \eqref{30_alt}, we formulate our safe approximation of Problem~\eqref{problem:basic2} with MIP techniques. For ease of notation, we denote an equation by a superscript $(\cdot)^s$, if it is expressed for a certain $s\in S$. The model is the content of the next theorem.

\begin{theorem}\label{Thm: MIP_onedim_safe}
	The feasible set of the optimization problem
	\begin{subequations}\label{Prob: MIP_onedim}
		\begin{align}
			\sup_{}~ & c(x^-,x^+, \tilde{x}) \\
			\text{s.t.}~ & \sum_{s \in S} \Big( \langle \left(1,-1, -\mu_+^s, \mu_-^s,-\sigma^{2,s}_+ + \mu_+^s \mu_-^s - \frac{(\delta_N^s)^2}{4} \right),y^s\rangle &&\hspace{-.3cm}- \sum_{\tau\in T_N^s} (\bar{\rho}^+_N)^s(\tau) \boundvariable{\tau}^s \Big) \geq b \label{Constr: discretized_dual_objective_geq0}\\
			& a^s \chi_{\tau}^s + z_{\tau}^s + p^s_{y^s}(\tau) \geq 0 && \forall \tau\in T_N^s, s\in S \label{Constr: discretized_purity_strengthened}\\
			& a^s \chi_{\tau}^s +z_{\tau}^s +p^s_{y^s}(\tau+\delta_N^s) \geq 0 &&\forall \tau \in T_N^s\setminus\{t^s_{\max}\}, s\in S \label{Constr: discretized_purity_strengthened_shifted}\\
			\label{27e}&\eqref{28}^s&&\forall s \in S\\
			&\eqref{29_alt}^s&&\forall s \in S\colon a^s > 0\\
			\label{27g}&\eqref{30_alt}^s&&\forall s \in S\colon a^s < 0\\
			& (x^-,x^+, \tilde{x}) \in C \label{Constr: x^-_x^+_in_P}\\
			& y^s \in \mathbb{R}^{5}_{\ge 0}, z^s \in \mathbb{R}_{\ge 0}^{T_N^s} && \forall s \in S. \label{Constr: MIP_onedim_nonneg_yz}
		\end{align}
	\end{subequations}
	projected to $x := (x^-, x^+, \tilde{x})$ is a subset of the feasible set of Problem~\eqref{problem:basic2}, and thus \eqref{Prob: MIP_onedim} is a safe approximation of \eqref{problem:basic2}. Moreover, this projection coincides with the feasible set of Problem \eqref{problem:safe_approx_outer}, so that \eqref{Prob: MIP_onedim} is an exact mixed-integer formulation of \eqref{problem:safe_approx_outer}.
\end{theorem}

\begin{proof}
	We consider $(x, \chi, \Delta, y, z)$ feasible for \eqref{Prob: MIP_onedim}, and show that $x$ is feasible for Problem~\eqref{problem:basic2}. As the objective functions of the two problems are identical, this implies directly that the former is a safe approximation of the latter.

	As $(x, \chi, \Delta)$ fulfills Constraints~\eqref{27e} to \eqref{27g}, we know that $a^s\chi^s_\tau \leq a^s\mathbbm{1}_{\tilde{I}(a^s, x^-, x^+)}(\tau)$ for all $\tau \in T_N^s,~s\in S$.
	Hence, $(x, y^s, z^s)$ is feasible for Problem~$\eqref{Prob: Dual_Purity_Constraint_ref}^s$.
	This implies that the expression
	\begin{equation*}
		\langle \left(1,-1, -\mu_+^s, \mu_-^s,-\sigma^{2,s}_+ + \mu_+^s \mu_-^s - \frac{(\delta_N^s)^2}{4} \right),y^s\rangle - \sum_{\tau\in T_N^s} (\bar{\rho}^+_N)^s(\tau) \boundvariable{\tau}^s
	\end{equation*}
	underestimates the optimal value of the inner problem $\eqref{problem:inner}^s$ for all $s\in S$.
	This suffices to prove that $x$ is feasible for Problem~\eqref{problem:basic2}, as the sum of the optimal values of the inner problems over $S$ is bounded from below by $b$.

	For the second statement, recall from Remark \ref{remark:alt_systems} that Systems \eqref{29_alt} and \eqref{30_alt} are equivalent to Systems \eqref{29} and \eqref{30}. For $a^s>0$, System \eqref{29} forces $\chi^s_\tau=0$ outside $\tilde{I}(a^s,x^-,x^+)$ and permits the value $1$ inside, while for $a^s<0$ System \eqref{30} forces $\chi^s_\tau=1$ inside and permits the value $0$ outside. In both cases the choice $\chi^s=\mathbbm{1}_{\tilde{I}(a^s,x^-,x^+)}$ is feasible, and since every admissible $\chi^s$ satisfies $a^s\chi^s_\tau \leq a^s\mathbbm{1}_{\tilde{I}(a^s,x^-,x^+)}(\tau)$, this choice makes \eqref{Constr: discretized_purity_strengthened} and \eqref{Constr: discretized_purity_strengthened_shifted} least restrictive, namely equal to the constraints of $\eqref{Prob: Dual_Purity_Constraint_ref}^s$. As the optimal values of $\eqref{problem:safe_approx_primal}^s$ and $\eqref{Prob: Dual_Purity_Constraint_ref}^s$ coincide by linear programming duality, $x$ is feasible for \eqref{Prob: MIP_onedim} after projection if and only if $\sum_{s\in S}\text{val}^s_N(x^-,x^+)\geq b$, which is \eqref{constr:safe_approx_outer}.
\end{proof}

Consequently, Corollary \ref{cor:safety}, Theorem \ref{thm:outer_convergence} and Corollary \ref{cor:enclosure} apply to \eqref{Prob: MIP_onedim}: its optimal value underestimates that of \eqref{problem:basic2}, converges to it as $\delta_N \rightarrow 0$ under the strict feasibility condition \eqref{eq:strict_feasibility}, and is complemented by an upper bound obtained from the same model with the right-hand side of \eqref{Constr: discretized_dual_objective_geq0} decreased by $\sum_{s\in S}\kappa^s_N$, at the cost of one additional solve of a model of identical size. The enclosure is evaluated for the application in Section~\ref{Sec:comp-results}.

\begin{remark}
	The tractability of the above problem depends on assumptions on the objective function $c$ as well as $C$. In particular, if we assume a linear objective function as well as $C$ being MIP representable, then \eqref{Prob: MIP_onedim} can be solved by modern MIP solvers. Convex MINLP problems such as MISDP or MIQP that contain distributionally robust constraints similar to \eqref{problem:inner} might still be in reach for state-of-the-art MINLP solvers if reformulated via Theorem \ref{Thm: MIP_onedim_safe}.
\end{remark}

\begin{remark}\label{remark:model_size}
	The number of binary variables of \eqref{Prob: MIP_onedim} is proportional to the number of grid points. All of them are determined by the two grid cells that contain the end points of the interval. The solver effectively branches only over pairs of grid cells, as the node counts in Table~\ref{Tab:runtime} show.
\end{remark}

\begin{remark}\label{remark:unique_binarification}
  Systems~\eqref{28}, \eqref{29_alt}, \eqref{30_alt} do not have to be set up for all $s \in S$ if the discretization grids $T_N^s$ are identical for some $s_1, s_2 \in S$ and $\sign(a^{s_1}) = \sign(a^{s_2})$. If $T_N^{s_1} = T_N^{s_2}$, we can omit System~$\eqref{28}^{s_2}$, using the $\Delta^{s_1}$ variables in $\eqref{29_alt}^{s_2}$ or $\eqref{30_alt}^{s_2}$ instead of $\Delta^{s_2}$. Further, we can omit System~$\eqref{29_alt}^{s_2}$/System~$\eqref{30_alt}^{s_2}$, using the $\chi^{s_1}$ variables in Constraints~\eqref{Constr: discretized_purity_strengthened} and \eqref{Constr: discretized_purity_strengthened_shifted} instead of $\chi^{s_2}$. In the extreme case where $T_N^s = T_N$ for all $s\in S$, we need the Systems~\eqref{28}, \eqref{29_alt} and \eqref{30_alt} at most twice.
  More generally, depending on the application, it may also be possible to exploit some problem characteristics, for example additional constraints on $x^+$ or $x^-$, to decrease the size of the MIP model further.
\end{remark}

\section{Computational Results}
\label{Sec:comp-results}
To evaluate the introduced approach, we test it on a
prototypical application in the setup of material design processes.
Our code is publicly available \cite{Dienstbier2025SafeMIPDR}.

\subsection{A DRO model for chromatographic fractionation}
A
fundamental and simultaneously challenging task in this active research field consists in the separation of a mixture of substances into its individual components, characterized by different
criteria. In this process, the
particle mixture flows along a so-called chromatographic column with
material-dependent velocities. Loosely speaking, while flowing along the column,
different materials can be separated. At the end of the process the concentration over time of each particle is detected and documented in the \emph{chromatogram}.
The application considered here consists in determining points in time when to collect
the materials leaving the column. The collected material then has to satisfy certain quality requirements. This collection process of separated materials is called \emph{fractionation}.
In this particular application, we consider polyethylene glycol
(PEG) molecules that shall be separated
with respect to their \emph{degree of polymerization} $s$, i.e., the (discrete) number $s$ of
monomeric units. Quality requirements then state that at least a certain
fraction of the separated material needs to have some specified degree.

Chromatographic processes are prone to uncertainties that already in
very simplified settings
can impact the separation results negatively. In particular, the \emph{residence time distributions} (RTDs), which distribute the time a PEG needs to pass the column may be uncertain themselves. In order to maintain quality
requirements even under uncertainty, robust
protection is sought for which not yet general approaches are available.

In our example, we use realistic settings from \cite{Supper2023a}. We denote $S$ as the set of all polymerization degrees within the injection and $S_w \subseteq S$ as the set of the polymerization degrees of the desired particles.
The uncertain RTDs are assumed to be (truncated) normal distributions, which is a standard assumption in practice, see e.g., \cite{supper2022separation}. Each distribution describes the degree of polymerization for one $s\in S$. We naturally assume that the mean $\mu^s$ and the standard deviation $\varianz{s}$ are uncertain.

Assuming a mixture of different PEGs, the aim is to set up
the separation such that as much share as possible is collected from one
desired PEG size. Thus, we need to find the time interval, i.e., a point in time $x^-$ where
to start and a point in time $x^+ > x^-$ where to end
fractionation. On the one hand, we wish to collect as much of the
desired PEG material. This amount is determined by the area under its
concentration distribution in the resulting chromatogram. As this area is strongly correlated with the quantity $x^+ - x^-$, we determine $x^-, x^+$ such that this difference is maximized, i.e., we aim at fractionating as long as possible.

On the other hand, quality requirements on the
end product need to be met. To this end, it is required that the percentage of the
desired PEG in the end product does not fall below a given bound. This means that
we require a purity of at least some predefined value $R\in [0,1)$.
Let $q_0$ denote the initial \emph{particle size density} (PSD),
i.e., the particle size density of the PEGs before entering the chromatograph. Here, we assume that the mixture is inserted into the chromatographic column via Dirac impulse, i.e., the entire mixture is inserted instantaneously. Now with $\u^s$ being the ambiguity set containing the plausible probability measures given by the RTD of polymerization degree $s$, the purity constraint for $x^+ - x^- > 0$ reads
\begin{equation*}
	R \leq \min_{\P^s \in \u^s \forall s\in S} \frac{\sum_{s \in S_w} q_0(s) \P^s([x^-, x^+]) }{\sum_{s \in S} q_0(s) \P^s\left([x^-, x^+]\right)}.
\end{equation*}
Assuming that the denominator is positive, which holds whenever the fractionation window carries mass of at least one species, this is equivalent to
\begin{equation*}
	R \leq \frac{\sum_{s \in S_w} q_0(s) \min_{\P \in \u^s} \P\left([x^-, x^+]\right) }{\sum_{s \in S_w} q_0(s) \min_{\P \in \u^s} \P\left([x^-, x^+]\right) + \sum_{s \in S \setminus S_w} q_0(s) \max_{\P \in \u^s} \P\left([x^-, x^+]\right)}
\end{equation*}
since, writing $N$ for the numerator and $N+D$ for the denominator, the ratio $\frac{N}{N+D}$ is non-decreasing in $N$ and non-increasing in $D$, while $N$ and $D$ depend on disjoint groups of measures, so that the worst case is attained by minimizing every term of $N$ and maximizing every term of $D$. Subtracting the left-hand side of the inequality and multiplying by the (positive) denominator, the constraint can be reformulated to
\begin{equation*}
	0 \leq \sum_{s\in S} \min_{\P\in \u^s} \left(\mathbbm{1}_{S_w}(s) - R\right)\initial{s}\P\left([x^-, x^+]\right).
\end{equation*}
Hence, the purity constraint matches with the constraint we have introduced in Section \ref{Sec: Problem Setting}. In formulas, we aim to solve:
\begin{subequations}\label{Prob: DRO_chromatography}
	\begin{align}
		\max_{x^-, x^+ \in T} &~x^+-x^-\\
		& \st\ 0 \leq \sum_{s\in S} \min_{\P\in \u^s}\left(\mathbbm{1}_{S_w}(s) - R\right)\initial{s}\P\left([x^-, x^+]\right) \label{Constr: purity_primal}
	\end{align}
\end{subequations}
with $x^-$ being the start, $x^+$ being the end of the fractionation. We assume that all ambiguity sets $\u^s$ are defined in accordance with Section~\ref{sec:ambiguity_set}.
Setting $a^s := \left(\mathbbm{1}_{S_w}(s) - R\right)\initial{s}$ for all $s \in S$, the problem has the desired structure and we can apply our framework, i.e., we can solve a safe approximation of the form \eqref{Prob: MIP_onedim}.

\subsection{Instances and results for realistic PEG data}
For our example we used process and optimization
parameters that are typical when working with PEGs. In the following, the most important ones are explained first. Depending on these, the nominal mean, as well as its minimum and maximum deviation can then be calculated.
For the nominal values, we start by noting that
a \emph{solvent} is inserted together with the
mixture that transports the latter through the column. In our case,
this is \emph{Acetonitrile} (ACN), where we denote its ratio by $r_\text{ACN}$.
The ACN ratio
impacts the so-called retention time that the PEGs need in general to flow through the column.
The so-called number of theoretical plates $NTP$ is a quantitative measure
of the separation efficiency of a chromatographic column. It influences the peak width through the standard deviation $\varianz{s} = \sqrt{\frac{(\mu^s)^2}{NTP}}$, so that the variance bound required in \eqref{Eq: Sec2_second_moment} is obtained by squaring this quantity. Now, using the numbers
from Table \ref{Tab: parameter} below, the mean $\mu^s$ and thus the standard deviation for each PEG can be determined using \cite{Supper2023a}.
\begin{table}[H]
\caption{Process and optimization parameters for example chromatogram}
    \centering
    \begin{tabular}{c|c}
        term  & value \\
        \hline
        \hline
        degrees of polymerization of the desired PEGs $s \in S_w$ & 32 \\
        \hline
        degrees of polymerization of all PEGs within the injection $s \in S$ & 30,31,32,33 \\
        \hline
        required purity $R$ & 0.95 \\
        \hline
        number of theoretical plates $NTP$ & $120\,000$ \\
        \hline
        ACN-ratio  $r_\text{ACN}$ & 0.25 \\
        \hline
    \end{tabular}
    \label{Tab: parameter}
\end{table}
In practice, the ACN-ratio is uncertain because of the imprecise pump
for injecting ACN.
We consider realistic uncertainties $r_\text{ACN} \in
[0.25 -\varepsilon_{\text{ACN}},0.25+\varepsilon_{\text{ACN}}]$
with $\varepsilon_{\text{ACN}} = 0.0042$ as this choice of parameters illustrates the effects of the safe approximation with regards to the presented application fairly well.
The uncertainty on $r_\text{ACN}$ leads to uncertainty in the
mean $\mu^s$.
Again with \cite{Supper2023a}, the corresponding means $\mu^s(r_{\text{ACN}})$ are calculated and result in the \emph{maximum and minimum retention
times} $\mu^s_-, \mu^s_+$, which are displayed in Table \ref{tab:inputvalues}.

\begin{table}[H]
\caption{(Bounds on) retention times in minutes calculated via \cite{Supper2023a}, using parameters in Table~\ref{Tab: parameter}.}
    \centering
    \begin{tabular}{|r||r||r r|}\hline
    $s$ & $\mu^s$%
    &$\mu^s_-$ & $\mu^s_+$
    \\
    \hline
    30 & 2.93
    & 2.865 & 2.998
    \\
    31 & 3.10
    & 3.030 & 3.179
    \\
    32 & 3.29
    & 3.209 & 3.377
    \\
    33 & 3.50
    & 3.403 & 3.593
    \\
\hline
    \end{tabular}
    \label{tab:inputvalues}
\end{table}

Furthermore, we choose the variance bound $\sigma^{2,s}_+$ of \eqref{Eq: Sec2_second_moment} as the largest variance $\sigma^s(r_{\text{ACN}})^2$ that occurs for $r_{\text{ACN}}\in \{0.25 -\varepsilon_{\text{ACN}},0.25,0.25+\varepsilon_{\text{ACN}}\}$, where $\sigma^s(r_{\text{ACN}})=\sqrt{\frac{\mu^s(r_{\text{ACN}})^2}{NTP}}$ denotes the standard deviation. Since $\mu^s$ is monotone in $r_{\text{ACN}}$ and $\sigma^s$ is monotone in $\mu^s$, this maximum over three sample points equals the maximum over the whole interval $[0.25-\varepsilon_{\text{ACN}},0.25+\varepsilon_{\text{ACN}}]$.
The corresponding envelope is defined by
\begin{equation*}
	\bar{\rho}^s(t)=\begin{cases}
		\rho_{\mu^s_-}(t) & \text{if } t \leq \mu^s_-\\
		\rho_{\mu^s_+}(t) & \text{if } t \geq \mu^s_+\\
		\max_{r_{\text{ACN}} \in \{0.25 -\varepsilon_{\text{ACN}},0.25,0.25+\varepsilon_{\text{ACN}}\}} \rho_{\mu^{s}(r_{\text{ACN}})} (\mu^{s}(r_{\text{ACN}})) & \text{otherwise,}
	\end{cases}
\end{equation*}
where $\rho_{\mu^{s}(r_{\text{ACN}})}$ denotes the density of the truncated normal distribution determined by $r_{\text{ACN}}$, truncated to $T$ and normalized before the envelope is formed, so that \eqref{Eq: envelope} indeed holds for the members of the family. The chromatograms resolve time only up to the grid width. The residence time densities and the envelope $\bar{\rho}^s$ are step functions on the grid, constant on every cell $[\tau,\tau+\delta_N)$ with the value given by the expression above.

We solved the safe approximation of Problem \eqref{Prob: DRO_chromatography} of the form \eqref{Prob: MIP_onedim} using Gurobi version 13.0.0 on a notebook with an Intel Core i7-7600U processor at $2.80$ GHz and $16$ GB of memory.
The number of variables and constraints depends strongly on the value of $\delta_N$, as shown in Table~\ref{Tab:runtime}. Up to $\delta_N=0.0001$ the running times stay below one minute, so that the calculation of robust fractionation times is tractable in practice, even though non-convex indicator functions are involved. The three finer discretizations were added to quantify the optimality gap, see below; they cost $137$ and $412$ seconds, and at $\delta_N=0.00001$ the branch-and-bound run reaches the dual bound after $107$ seconds but does not produce a matching incumbent within the time limit. In that row, the difficulty lies in finding a feasible point; the relaxed model of the same size solves to proven optimality in $1425$ seconds. Building the model in \texttt{Python} takes between $3.7$ and $42.6$ seconds and the peak memory reaches $4.9$ of the $15.4$ available gigabytes, both linear in the number of grid cells as Remark \ref{remark:model_size} predicts.

\begin{table}[H]
\caption{Model sizes and running times for solving the safe approximation of \eqref{Prob: DRO_chromatography} with different $\delta_N$. $^{\ddagger}$ this run hit a time limit of $2400$ seconds; its optimal value is nevertheless determined, see the discussion after Table \ref{Tab: upper_bound}.}
    \centering
    \small
    \begin{tabular}{|r|r|r|r|r|r|r|}\hline
        $\delta_N$ (min) & variables & continuous & binary & constraints & time (s) & nodes \\
        \hline
        0.01 & 763 & 348 & 415 & 1064 & $<1$ & 2\\
        0.005 & 1510 & 680 & 830 & 2143 & $<1$ & 2\\
        0.002 & 3733 & 1668 & 2065 & 5354 & $<1$ & 2\\
        0.001 & 7459 & 3324 & 4135 & 10736 & 1.53 & 59\\
        0.0005 & 14893 & 6628 & 8265 & 21474 & 5.88 & 232\\
        0.0002 & 37204 & 16544 & 20660 & 53701 & 12.00 & 2\\
        0.0001 & 74401 & 33076 & 41325 & 107430 & 44.84 & 8\\
        0.00005 & 148777 & 66132 & 82645 & 214862 & 137.04 & 6\\
        0.00002 & 371905 & 165300 & 206605 & 537158 & 411.72 & 2\\
        0.00001 & 743785 & 330580 & 413205 & 1074318 & $2400^{\ddagger}$ & 42\\
        \hline
    \end{tabular}
    \label{Tab:runtime}
\end{table}

Next, we compare the obtained robust fractionation times with that of the nominal fractionation.

In Figure \ref{fig:chromatograms_uncertainty}, the y-axis of a chromatogram shows its output signal (an electric impulse) that is proportional to the particle concentration.
Each of the four gray or red sharp peaks corresponds to the nominal PSD of one PEG after leaving the column. The narrow red peak is the nominal PSD of the desired PEG.
Under the considered uncertainty sets, the peaks lie below the envelope functions, drawn as dashed lines. The filled areas are the worst-case densities of the four adversarial problems for the fractionation interval shown, so that the figure also shows the worst cases chosen by the adversary.
\begin{figure}[ht]
	\centering

	\begin{subfigure}[t]{\textwidth}
		\centering
		\includegraphics[width=\textwidth]{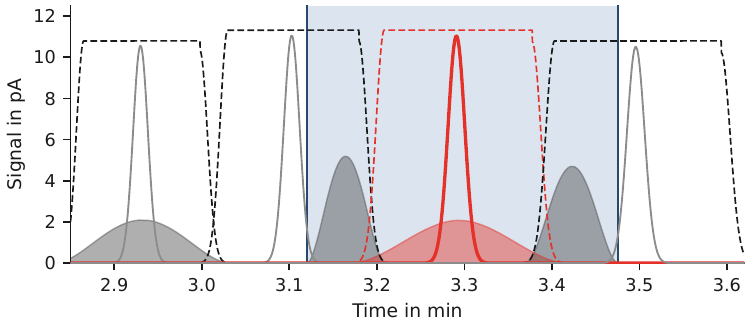}
		\caption{%
			Nominal fractionation interval $[3.1198, 3.4760]$: two contaminants can place their entire mass inside the interval, and the guaranteed worst-case purity drops to $33.8\%$.
		}
		\label{fig:no_improvement_too_small}
	\end{subfigure}

	\vspace{2mm}

	\begin{subfigure}[t]{\textwidth}
		\centering
		\includegraphics[width=\textwidth]{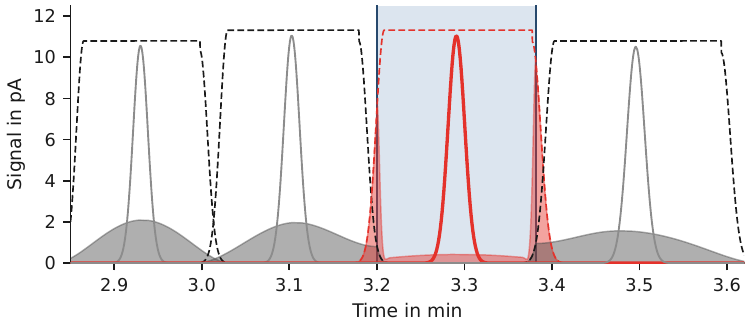}
		\caption{%
			Robust fractionation interval $[3.2004, 3.38158]$ with a guaranteed worst-case purity of $95.0002\%$: the contaminants are pushed out of the interval, at the price of an interval shorter by a factor of two, and the desired species keeps $47.1\%$ of its mass inside.
		}
		\label{fig:no_improvement_too_big}
	\end{subfigure}

	\caption{Nominal and robust chromatograms at $\delta_N=2\cdot 10^{-5}$. Solid lines are the nominal densities, dashed lines the envelopes, the shaded band is the fractionation interval, and the filled areas are the worst-case densities of the four adversarial problems for that interval, all in the units of the measured signal. The adversarial problems have many optimal solutions; shown is the one of least curvature.}
	\label{fig:chromatograms_uncertainty}
\end{figure}
Depending on the realization of the uncertainty the peaks can overlap. The shaded band shows the fractionation interval.

It turns out that for these uncertainties, robust protection indeed influences the fractionation time interval. The nominal fractionation interval is about twice the length of that of the robust fractionation interval, so that in the robust setting less end product is obtained.

However, we observe that the purity of the nominal solution drops to $33.8\%$ if every contaminant traverses the column at its least favorable admissible speed while the desired species behaves nominally, which is far below the target purity of $95\%$.
Hence, implementing the nominal fractionation interval leads to unusable end product so that everything needs to be discarded.
For the robust fractionation times the two evaluations differ. Under the sampled evaluation just described the end product has a purity of $97.6\%$; over the whole ambiguity set \eqref{Eq: ambiguity_set} its purity is guaranteed to be at least $95.0002\%$ by Theorem \ref{Thm: MIP_onedim_safe}.

\subsection{Quality of the safe approximation}\label{sec:quality}
The previous subsection shows that robust protection is needed for this application. We now quantify how far the safe approximation is from the original Problem \eqref{Prob: DRO_chromatography}.

\emph{Upper bound and optimality gap.} Problem \eqref{Prob: DRO_chromatography} cannot be solved directly: for fixed $x^-,x^+$ its adversarial problems are linear programs over an infinite-dimensional cone of measures with a continuum of constraints and a discontinuous integrand, and the outer problem is a nonconvex bilevel problem over their optimal values. Instead, we bound the optimal value from above. By Corollary \ref{cor:enclosure}, solving the mixed-integer program a second time with the right-hand side of \eqref{Constr: discretized_dual_objective_geq0} decreased by $\sum_{s\in S}\kappa^s_N$ turns it into a relaxation of \eqref{Prob: DRO_chromatography}, so that its optimal value $\text{val}^{\uparrow}_N$ bounds the true one from above, while the safe approximation itself bounds it from below. Table \ref{Tab: upper_bound} reports both bounds and the resulting optimality gap for the seven finest discretizations of Table \ref{Tab:runtime}. The gap is guaranteed but conservative, since $\kappa^s_N$ is a worst-case bound over all admissible envelopes and all $x^-,x^+$: at $\delta_N=0.0001$ the sum $\sum_{s\in S}\kappa^s_N$ still amounts to $7\%$ of $\sum_{s\in S}|a^s|$. It shrinks in proportion to $\delta_N$, as Theorem \ref{thm:inner_convergence} predicts, and with it the gap: from $0.352$ minutes at $\delta_N=0.001$ to $0.0216$ minutes at $\delta_N=0.0001$, i.e.\ $12.0\%$ of the optimal value of the safe approximation, and further to $0.0062$ minutes, i.e.\ $3.4\%$, at $\delta_N=0.00001$.

The mixed-integer program of the last row hit the time limit. Its optimal value is nevertheless known exactly. The objective counts grid cells, the dual bound of the truncated run is $18127$ cells, and the window of exactly $18127$ cells, $[3.20028,\,3.38155]$, is certified feasible by the inner linear programs, which implement the same feasibility test as the model; hence $\text{val}_N=0.18127$ for that row. Every $\text{val}_N$ in the table is a valid lower bound and every $\text{val}^{\uparrow}_N$ a valid upper bound on the optimal value of \eqref{Prob: DRO_chromatography}, so the sharpest pair over all rows may be combined: here both come from the finest row, and $\text{val}^{*}\in[0.18127,\,0.18750]$.
\begin{table}[H]
\caption{Lower and upper bound on the optimal value of \eqref{Prob: DRO_chromatography} according to Corollary \ref{cor:enclosure}, and the resulting relative optimality gap $(\text{val}^{\uparrow}_N-\text{val}_N)/\text{val}_N$ in percent. Fractionation times in minutes; $\sum_{s\in S}\kappa^s_N$ is stated in the units of the purity constraint \eqref{Constr: purity_primal}, in which $\sum_{s\in S}|a^s|=71.4$. $^{\ddagger}$ this run hit the time limit; its optimal value is determined by a dual bound and a primal certificate, see the text.}
	\centering
	\begin{tabular}{|r|r|r|r|r|}\hline
		$\delta_N$ (min) & $\sum_{s\in S}\kappa^s_N$ & $\text{val}_N$ & $\text{val}^{\uparrow}_N$ & gap (\%) \\
		\hline
		0.001 & 44.38 & 0.1220 & 0.4740 & $\geq$100\\
		0.0005 & 23.84 & 0.1695 & 0.3685 & $\geq$100\\
		0.0002 & 10.01 & 0.1792 & 0.2092 & 16.7\\
		0.0001 & 5.09 & 0.1804 & 0.2020 & 12.0\\
		0.00005 & 2.57 & 0.1809 & 0.1963 & 8.5\\
		0.00002 & 1.03 & 0.1812 & 0.1906 & 5.2\\
		0.00001 & 0.52 & $0.1813^{\ddagger}$ & 0.1875 & 3.4\\
		\hline
	\end{tabular}
	\label{Tab: upper_bound}
\end{table}

\section{Conclusion}
\label{Sec: Conclusion}

In this paper, we have presented an approach for safely approximating distributionally robust optimization problems where non-convex one-dimensional indicator functions are allowed.
We have shown that a suitably discretized formulation yields a safe approximation that is a mixed-integer linear program whenever the objective function is linear and the remaining feasible set is mixed-integer representable.
For moment-constrained ambiguity sets, we have proven that the approximation of the inner problem converges to the true inner problem, with an error bound that is uniform in the upper-level decisions. This uniformity carries the result over to the outer problem: the optimal value of the safe approximation converges to the optimal value of the original distributionally robust problem whenever the latter admits strictly feasible near-optimal points, and a pathological example shows that this assumption cannot be omitted. In addition, the same bound provides an enclosure of the true optimal value that is computed by a single additional solve of a model of the same size.
The fact that indicator functions can be included in the model pushes the applicability of duality-based reformulations of distributional robustness beyond convexity.

In future work, one might build on the present work on indicator functions to consider DRO constraints that allow more complex nonlinear interactions $v(x,t)$ between the decision and random variables. Moreover, the definition of ambiguity sets could be further expanded by incorporating additional classes of constraints. Nevertheless, the computational results presented in this paper demonstrate that the formulation introduced here already meets the imposed quality requirement for a relevant application, and that the optimality gap with respect to the original distributionally robust problem can be certified, here as at most $3.4\%$.
\section*{Data Availability}
The data and the code that support the findings of this study are archived on Zenodo, see \cite{Dienstbier2025SafeMIPDR}. The digital object identifier given there resolves to all versions of the archive; the results reported in Section \ref{Sec:comp-results} were produced with version $0.3.0$.
\section*{Declaration on the Use of AI Tools}
The originally submitted version of this article was written without the use of AI tools. During the revision, the authors used Claude Code (Anthropic) to draft the statements and proofs of Section \ref{sec:outer}, to revise text and proofs in Sections 2 to 4, to design and carry out the computational study reported in Section \ref{sec:quality}, and to prepare the responses to the referees. All content was subsequently verified and, where necessary, edited by the authors, who take full responsibility for it.

\section*{Acknowledgments}
\label{sec:acknowledgements}
We thank Dominique Fahrnbach for his intensive work on a
Wasserstein-based approach as part of his
Master's thesis. We thank
Malte Kaspereit and Malvina Supper for many
inspiring discussions on particle separation processes, in particular
for handing over the parameter values for Table \ref{Tab:
  parameter} and the formulas for calculating the numbers in Table \ref{tab:inputvalues}.
The authors are indebted to the anonymous referees whose comments and suggestions greatly improved the manuscript.
The paper is funded by the Deutsche Forschungsgemeinschaft (DFG, German Research Foundation) - Project-ID 416229255 - SFB 1411.

\appendix
\section{Notation}\label{sec:notation}
The following table collects the symbols used in this article, in the order in which they are introduced. Where a quantity carries a species index $s$, the index is suppressed whenever a single adversarial problem is considered, as announced at the end of Section \ref{Sec: Problem Setting}.

\begin{longtable}{@{}lp{0.63\textwidth}l@{}}
	\toprule
	symbol & meaning & introduced in\\
	\midrule
	\endfirsthead
	\toprule
	symbol & meaning & introduced in\\
	\midrule
	\endhead
	\multicolumn{3}{@{}l}{\emph{Problem data}}\\[0.3ex]
	$S$, $S_w$ & finite index set of the uncertain quantities; subset of the desired ones & \eqref{problem:basic2}, Sec.~\ref{Sec:comp-results}\\
	$T^s = [t^s_{\min}, t^s_{\max}]$ & compact interval, domain of the $s$-th random variable $t^s$ & \eqref{problem:basic2}\\
	$C$ & feasible set of the upper-level decisions, $C\subseteq \R^k$ & \eqref{problem:basic2}\\
	$c$ & objective function of the upper-level problem & \eqref{problem:basic2}\\
	$b$ & right-hand side of the distributionally robust constraint & \eqref{problem:basic2}\\
	$a^s$ & coefficient of the $s$-th indicator function & \eqref{problem:basic2}\\
	$v$ & integrand of the distributionally robust constraint & \eqref{Eq: DRO_constraint}, \eqref{Eq: v_one-dim}\\
	$x^-,x^+$ & ends of the interval on which the indicators are supported & \eqref{problem:basic2}\\
	$\tilde{x}$ & remaining upper-level decisions, absent from the constraint & \eqref{problem:basic2}\\
	$q_0$, $R$ & initial particle size density; required purity & Sec.~\ref{Sec:comp-results}\\
	$r_{\text{ACN}}$, $\varepsilon_{\text{ACN}}$, $NTP$ & acetonitrile ratio, its uncertainty, number of theoretical plates & Sec.~\ref{Sec:comp-results}\\[0.6ex]
	\multicolumn{3}{@{}l}{\emph{Measures and ambiguity sets}}\\[0.3ex]
	$\P$, $\P^s$ & uncertain probability measure of $t$, of $t^s$ & \eqref{Eq: DRO_constraint}, \eqref{problem:basic2}\\
	$\u^s$ & ambiguity set of the $s$-th random variable & \eqref{Eq: ambiguity_set}\\
	$\Lcal$ & Lebesgue measure & \eqref{Eq: ambiguity_set}\\
	$\rho_{\P}$ & density of $\P$ with respect to $\Lcal$ & \eqref{Eq: envelope}\\
	$\bar{\rho}^s$, $\rho^{max}$ & envelope of the densities; an upper bound for it & \eqref{Eq: envelope}, Thm.~\ref{thm:inner_convergence}\\
	$\mu^s_-,\mu^s_+$, $\mu$ & bounds on the first moment; their midpoint & \eqref{Eq: Sec2_first_moment}\\
	$\sigma^{2,s}_+$ & bound on the variance, in the square of the unit of $t$ & \eqref{Eq: Sec2_second_moment}\\
	$\Ccal(T)$, $\Mcal(T)$, $\Mcal(T)_{\geq 0}$ & continuous functions, signed and non-negative Radon measures & Sec.~\ref{sec:inner}\\
	$\langle \cdot,\cdot\rangle$ & duality pairing $\langle g,\P\rangle=\int_T g\,d\P$ & Sec.~\ref{sec:inner}\\[0.6ex]
	\multicolumn{3}{@{}l}{\emph{Discretization}}\\[0.3ex]
	$K^s_N$, $\delta^s_N$, $\eta_N$ & number of grid cells; grid width $(t^s_{\max}-t^s_{\min})/K^s_N$; width of the continuous smearing of an indicator & Sec.~\ref{sec:inner}\\
	$T^s_N$, $\tau$ & grid $\{t^s_{\min} + k\delta^s_N \mid k=0,\ldots,K_N\}$; a grid point & Sec.~\ref{sec:inner}\\
	$\bar{T}$ & length $t_{\max}-t_{\min}$ of $T$ & Thm.~\ref{thm:inner_convergence}\\
	$\mathbbm{1}_A$ & indicator function of the set $A$ & \eqref{Eq: v_one-dim}\\
	$\mathbbm{1}^{\eta}_{[\tau,\tilde\tau]}$ & any continuous function between $\mathbbm{1}_{[\tau,\tilde\tau]}$ and $\mathbbm{1}_{[\tau-\eta,\tilde\tau+\eta]}$ & Sec.~\ref{sec:inner}\\
	$\mathbbm{1}^{c}_{[x^-,x^+]}$ & continuous under\-estimator of $a\mathbbm{1}_{[x^-,x^+]}$, depending on the sign of $a$ & Sec.~\ref{sec:inner}\\
	$\tau^-_N,\tau^+_N$ & points of $t_{\min}+\delta_N\mathbbm{Z}$ delimiting that under\-estimator & \eqref{taudef}, \eqref{taudef2}\\
	$\bar{\rho}^+_N(\tau)$ & envelope mass of the cell at $\tau$ & \eqref{Eq: envelope_discretized}\\
	$I(a,x^-,x^+)$, $\tilde{I}(a,x^-,x^+)$ & interval of the finite dual constraint; its weakened form & Sec.~\ref{sec:inner}\\[0.6ex]
	\multicolumn{3}{@{}l}{\emph{Variables of the models}}\\[0.3ex]
	$y^s\in\R^5_{\geq 0}$ & duals of the normalization and the two moment constraints & \eqref{Prob: Dual_Purity_Constraint}\\
	$z^s_\tau\geq 0$ & duals of the discretized envelope constraints & \eqref{Prob: Dual_Purity_Constraint}\\
	$w^-_\tau,w^+_\tau$ & weights of the re-dualized adversary at $\tau$ and $\tau+\delta_N$ & \eqref{problem:safe_approx_primal}\\
	$\Delta^-_\tau,\Delta^+_\tau$ & binaries marking the cells that contain $x^-$ and $x^+$ & \eqref{28}\\
	$\chi^s_\tau$ & binaries encoding $\mathbbm{1}_{\tilde{I}(a^s,x^-,x^+)}$ & \eqref{29}, \eqref{30}\\[0.6ex]
	\multicolumn{3}{@{}l}{\emph{Quantities of the analysis}}\\[0.3ex]
	$p_y$, $t_y$, $f_{y,z}$ & polynomial part of the dual constraint, its minimizer over $T$, the full left-hand side & Sec.~\ref{sec:inner}\\
	$\kappa^s_N$ & error bound of Theorem \ref{thm:inner_convergence} for the $s$-th adversarial problem & \eqref{eq:error_bound}\\
	$\P_0$, $\P_1$ & nominal measure; measure constructed in the proof of Theorem \ref{thm:inner_convergence} & Thm.~\ref{thm:inner_convergence}\\
	$\text{val}^s$, $\text{val}^s_N$ & optimal value of the $s$-th adversarial problem and of its approximation & Sec.~\ref{sec:outer}\\
	$\text{val}^{*}$, $\text{val}_N$, $\text{val}^{\uparrow}_N$ & optimal value of \eqref{problem:basic2}, of the safe approximation, and of the relaxation of Corollary \ref{cor:enclosure} & Sec.~\ref{sec:outer}\\
	\bottomrule
\end{longtable}

\printbibliography

\end{document}